\documentclass[MSNbibl,number,citesort,seceqn,dvips]{arxbj}
\usepackage{graphicx}

\aid{0}
\volume{19}
\issue{2}
\pubyear{2013}
\firstpage{521}
\lastpage{547}
\doi{10.3150/11-BEJ410}

\makeatletter
\newtheorem{theorem}{Theorem}
\newtheorem{corollary}{Corollary}
\newtheorem{lemma}{Lemma}
\newproclaim{condition}{Condition}
\newremark{remark}{Remark}[section]

\renewcommand{\hat}{\widehat}

\newcommand{\cc}{\mathbf{c}}
\newcommand{\Gn}{\mathbb{G}_n}
\newcommand{\Pn}{\mathbb{P}_n}
\newcommand{\En}{\mathbb{E}_n}
\newcommand{\Ep}{\mathrm{E}}
\def\RR{\mathbb{R}}
\def\cc{{\bar c}}
\def\ii{{\bullet}}
\makeatother

\begin{document}
\begin{frontmatter}

\title{Least squares after model selection in high-dimensional sparse models}
\runtitle{Least squares after model selection}

\begin{aug}
\author[1]{\fnms{Alexandre} \snm{Belloni}\corref{}\thanksref{1}\ead[label=e1]{abn5@duke.edu}} \and
\author[2]{\fnms{Victor} \snm{Chernozhukov}\thanksref{2}\ead[label=e2]{vchern@mit.edu}}
\runauthor{A. Belloni and V. Chernozhukov} 
\address[1]{100 Fuqua Drive, Durham, North Carolina 27708, USA.
\printead{e1}}
\address[2]{50 Memorial Drive, Cambridge, Massachusetts 02142, USA. \printead{e2}}
\end{aug}

\received{\smonth{4} \syear{2010}}
\revised{\smonth{6} \syear{2011}}

%
\begin{abstract}
In this article we study post-model selection estimators that apply
ordinary least squares (OLS) to the model selected by first-step
penalized estimators, typically {Lasso}. It is well known that {Lasso}
can estimate the nonparametric regression function at nearly the oracle
rate, and is thus hard to improve upon. We show that the {OLS}
post-{Lasso} estimator performs at least as well as {Lasso} in terms of
the rate of convergence, and has the advantage of a smaller bias.
Remarkably, this performance occurs even if the {Lasso}-based model
selection ``fails'' in the sense of missing some components of the
``true'' regression model. By the ``true'' model, we mean the best
$s$-dimensional approximation to the nonparametric regression function
chosen by the oracle. Furthermore, {OLS} post-{Lasso} estimator can
perform strictly better than {Lasso}, in the sense of a strictly faster
rate of convergence, if the {Lasso}-based model selection correctly
includes all components of the ``true'' model as a subset and also
achieves sufficient sparsity. In the extreme case, when {Lasso}
perfectly selects the ``true'' model, the {OLS} post-{Lasso} estimator
becomes the oracle estimator. An important ingredient in our analysis
is a new sparsity bound on the dimension of the model selected by
{Lasso}, which guarantees that this dimension is at most of the same
order as the dimension of the ``true'' model. Our rate results are
nonasymptotic and hold in both parametric and nonparametric models.
Moreover, our analysis is not limited to the {Lasso} estimator acting
as a selector in the first step, but also applies to any other
estimator, for example, various forms of thresholded {Lasso}, with good
rates and good sparsity properties. Our analysis covers both
traditional thresholding and a new practical, data-driven thresholding
scheme that induces additional sparsity subject to maintaining a
certain goodness of fit. The latter scheme has theoretical guarantees
similar to those of {Lasso} or {OLS} post-{Lasso}, but it dominates
those procedures as well as traditional thresholding in a wide variety
of experiments.
\end{abstract}

%

\end{frontmatter}

\section{Introduction}\label{intr}
In this work, we study post-model selection estimators for
linear regression in high-dimensional sparse models ({hdsms}). In
such models, the overall number of regressors $p$ is very large,
possibly much larger than the sample size $n$.
However, there are $s=\mathrm{o}(n)$ regressors that capture most of the impact
of all covariates\vadjust{\goodbreak} on the response variable. {hdsms} \cite
{CandesTao2007,MY2007} have emerged to
deal with many new applications arising in biometrics, signal
processing, machine learning, econometrics, and other areas of data
analysis where high-dimensional data sets have become widely available.

Several authors have investigated estimation of {hdsms}, focusing
primarily on mean regression with the $\ell_1$-norm acting as a penalty
function \cite
{BickelRitovTsybakov2009,BuneaTsybakovWegkamp2006,BuneaTsybakovWegkamp2007,BuneaTsybakovWegkamp2007b,CandesTao2007,Koltchinskii2009,MY2007,vdGeer,Wainright2006,ZhangHuang2006}.
The results of \cite
{BickelRitovTsybakov2009,BuneaTsybakovWegkamp2006,BuneaTsybakovWegkamp2007,BuneaTsybakovWegkamp2007b,Koltchinskii2009,MY2007,Wainright2006,ZhangHuang2006}
demonstrate the fundamental result that $\ell_1$-penalized least
squares estimators achieve the rate $\sqrt{s/n} \sqrt{\log p}$, which
is very close to the oracle rate $\sqrt{s/n}$ achievable when the true
model is known.~\cite{Koltchinskii2009,vdGeer} demonstrated a similar
fundamental result on the excess forecasting error loss under both
quadratic and nonquadratic loss functions. Thus the estimator can be
consistent and can have excellent forecasting performance even under
very rapid, nearly exponential growth of the total number of regressors
$p$. In addition,~\cite{BC-SparseQR} investigated the
$\ell_1$-penalized quantile regression process and obtained similar
results. See \cite
{BickelRitovTsybakov2009,BuneaTsybakovWegkamp2006,BuneaTsybakovWegkamp2007,BuneaTsybakovWegkamp2007b,FanLv2006,Lounici2008,LouniciPontilTsybakovvandeGeer2009,RosenbaumTsybakov2008}
for many other interesting developments and a detailed review of the
existing literature.

In this article we derive theoretical properties of post-model
selection estimators that apply ordinary least squares ({OLS}) to the
model selected by first-step penalized estimators, typically {Lasso}.
It is well known that {Lasso} can estimate the mean regression function
at nearly the oracle rate, and thus is hard to improve on. We show that
{OLS} post-{Lasso} can perform at least as well as {Lasso} in terms of
the rate of convergence, and has the advantage of a smaller bias. This
nice performance occurs even if the {Lasso}-based model selection
``fails'' in the sense of missing some components of the ``true''
regression model. (By the ``true'' model, we mean the best
$s$-dimensional approximation to the regression function chosen by the
oracle.) The intuition for this result is that {Lasso}-based model
selection omits only those components with relatively small
coefficients. Furthermore, {OLS} post-{Lasso} can perform better than
{Lasso} in the sense of a strictly faster rate of convergence, if the
{Lasso}-based model correctly includes all components of the ``true''
model as a subset and is sufficiently sparse. Of course, in the extreme
case, when {Lasso} perfectly selects the ``true'' model, the {OLS}
post-{Lasso} estimator becomes the oracle estimator.

Importantly, our rate analysis is not limited to the {Lasso} estimator
in the first step, but applies to a wide variety of other first-step
estimators, including, for example, thresholded {Lasso}, the Dantzig
selector, and their various modifications. We provide generic rate
results that cover any first-step estimator for which a rate and a
sparsity bound are available. We also present a generic result from
using thresholded {Lasso} as the first-step estimator, where
thresholding can be performed by a traditional thresholding scheme
({t-}{Lasso}) or by a new fitness-thresholding scheme that we introduce
here ({fit}-{Lasso}). The new thresholding scheme induces additional
sparsity subject to maintaining a certain goodness of fit in the sample
and is completely data-driven. We show that {OLS} post-{fit} {Lasso}
estimator performs at least as well as the {Lasso} estimator, but can
be strictly better under good model selection properties. 

Finally, we conduct a series of computational experiments and find that
the results confirm our theoretical findings. Figure~\ref{Fig:Main}
provides a brief graphical summary of our theoretical results, showing
how the empirical risk of various estimators change with the signal
strength $C$ (coefficients of relevant covariates are set equal to
$C$). For very low signal levels, all estimators perform similarly.
When the signal strength is intermediate, {OLS} post-{Lasso} and {OLS}
post-{fit} {Lasso} significantly outperform {Lasso} and the {OLS}
post-{t} {Lasso} estimators. However, we find that the {OLS} post-{fit}
{Lasso} outperforms {OLS} post-{Lasso} whenever {Lasso} does not
produce very sparse solutions, which occurs if the signal strength
level is not low. For large levels of signal, {OLS} post-{fit} {Lasso}
and {OLS} post-{t} {Lasso} perform very well, improving on {Lasso} and
{OLS} post-{Lasso}. Thus, the main message here is that {OLS}
post-{Lasso} and {OLS} post-{fit} {Lasso} perform at least as well as
{Lasso} and sometimes a lot better.
%
%
\begin{figure}

\includegraphics{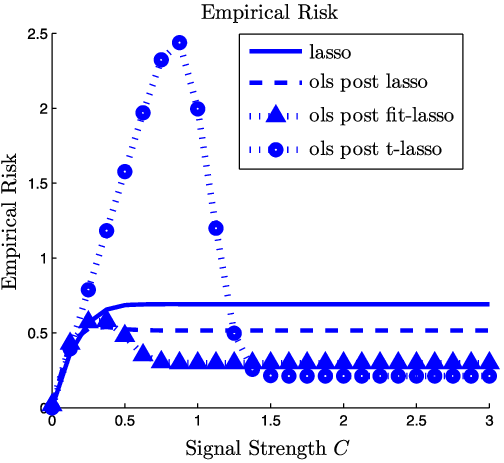}

\caption{This figure plots the performance of the
estimators listed in the text under the equicorrelated design for the
covariates $x_i \sim N(0,\Sigma)$, $\Sigma_{jk} = 1/2$ if $j\neq k$.
The number of regressors is $p=500$, and the sample size is $n=100$
with 1000 simulations for each level of signal strength $C$. In each
simulation, there are $5$ relevant covariates whose coefficients are
set equal to the signal strength $C$, and the variance of the noise is
set to $1$.}\label{Fig:Main}
\end{figure}

To the best of our knowledge, this article is the first to establish
the aforementioned rate results on {OLS} post-{Lasso} and the proposed
{OLS} post-fitness-thresholded {Lasso} in the mean regression problem.
Our analysis builds on the ideas of~\cite{BC-SparseQR}, who established
the properties of postpenalized procedures for the related, but
different problem of median regression. Our analysis also builds on the
fundamental results of~\cite{BickelRitovTsybakov2009} and the other
works cited above that established the properties of the first-step
{Lasso}-type estimators. An important ingredient in our analysis is a
new sparsity bound on the dimension of the model selected by {Lasso},
which guarantees that this dimension is at most of the same order as
the dimension of the ``true'' model. This result builds on some
inequalities for sparse eigenvalues and reasoning previously given by
\cite{BC-SparseQR} in the context of median regression. Our sparsity
bounds for {Lasso} improve on the analogous bounds of \cite
{BickelRitovTsybakov2009} and are comparable to the bounds of \cite
{ZhangHuang2006} obtained under a larger penalty level. We also rely on
the maximal inequalities of~\cite{ZhangHuang2006} to provide primitive
conditions for the sharp sparsity bounds to hold.

The article is organized as follows. Section~\ref{sec2} reviews the model and
discusses the estimators. Section~\ref{sec3} revisits some benchmark results of
\cite{BickelRitovTsybakov2009} for {Lasso}, allowing for a data-driven
choice of penalty level, develops an extension of model selection
results of~\cite{Lounici2008} to the nonparametric case, and derives a
new sparsity bound for {Lasso}. Section~\ref{Sec:PostModel} presents a generic rate
result on {OLS} post-model selection estimators. Section~\ref{sec5} applies the
generic results to the {OLS} post-{Lasso} and the {OLS} post-thresholded
{Lasso} estimators. The \hyperref[app]{Appendix} contains main proofs, and the
supplemental article~\cite{BC-PostL1-SA} contains auxiliary proofs, as
well as the results of our computational experiments.
\subsection*{Notation}
When making asymptotic statements, we assume that $n
\to\infty$ and $p=p_n \to\infty$, and also allow for $s=s_n \to
\infty$. In what follows, all parameter values  are indexed by the
sample size $n$, but we omit the index whenever this omission will not
cause confusion. We use the notation $(a)_+ = \max\{a,0\}$, $a \vee b =
\max\{ a, b\}$, and $a \wedge b = \min\{ a , b \}$. The $\ell_2$-norm
is denoted by $\|\cdot\|$, the $\ell_1$-norm is denoted by $\|\cdot\|
_1$, the $\ell_\infty$-norm is denoted by $\|\cdot\|_\infty$,
and the $\ell_0$-norm $\|\cdot\|_0$ denotes the number of nonzero
components of a vector. Given a vector $\delta\in\RR^p$ and
a set of indices $T \subset\{1,\ldots,p\}$,
%
we denote by $\delta_T$ the vector in $\mathbb R^p$ in which $\delta_{Tj} = \delta_j$
if $j\in T$ and $\delta_{Tj}= 0$
if $j \notin T$. The cardinality of $T$ is denoted by $|T|$.
Given a covariate vector $x_i \in\RR^p$, we let $x_i[T]$ denote
the vector $\{x_{ij}, j \in T\}$. The symbol $\Ep[\cdot]$ denotes the
expectation. We also use standard empirical process notation
\[
\En[f(z_\ii)] := \sum_{i=1}^n f(z_i)/n \quad\mbox{and}\quad \Gn(f(z_\ii))
:= \sum_{i=1}^n \bigl( f(z_i) - \Ep[f(z_i)] \bigr)/\sqrt{n}.
\]
We denote the $L^2(\Pn)$ norm by $\|f\|_{\Pn,2}=(\En[f_\ii^2])^{1/2}$.
Given covariate values $x_1,\ldots,x_n$, we define the prediction norm
of a vector $\delta\in\RR^p$ as $\|\delta\|_{2,n} = \{\En[(x_\ii
'\delta)^2]\}^{1/2}=(\delta'\En[x_\ii x_\ii']\delta)^{1/2}$.
We use the notation $a \lesssim b$ to denote $a \leq C b$ for some
constant $C>0$ that does not depend on $n$ (and thus does not depend on
quantities indexed by $n$ like $p$ or~$s$), and $a\lesssim_P b$ to
denote $a=O_P(b)$. For an event $A$, we say that $A$ wp $\to$ 1 when
$A$ occurs with probability approaching 1 as $n$ increases. In
addition, we write $\bar c = (c+1)/(c-1)$ for a chosen constant $c>1$.
%
\section{Setting, estimators, and conditions}\label{sec2}
\subsection{Setting}\label{conm}
\begin{condition*}[M]We have data $\{(y_i,z_i),
i=1,\ldots,n\}$ such that for each $n$,
%
%
\begin{equation}\label{equation model}
y_i = f(z_i) + \epsilon_i, \qquad\epsilon_i \sim N(0, \sigma^2), i=1,\ldots,n,
\end{equation}
where $y_i$ are the outcomes, $z_i$ are vectors of fixed regressors,
and $\epsilon_i$ are i.i.d. errors. Let $P(z_i)$
be a given $p$-dimensional dictionary of technical regressors with
respect $z_i$, that is, a p-vector of transformation of $z_i$, with
components
\[
x_i:=P(z_i)
\]
of the dictionary normalized so that
\[
\En[x_{\ii j}^2] = 1\qquad\mbox{for } j=1,\ldots,p.
\]
In making
asymptotic statements, we assume that $n \to\infty$ and $p=p_n \to
\infty$,
and that all parameters of the model are implicitly indexed by $n$.
\end{condition*}

We would like to estimate the nonparametric regression function $f$ at
the design points, namely the values $f_i = f(z_i)$ for $i=1,\ldots,n$.
To set up the estimation and define a performance benchmark, we
consider the following oracle risk minimization
program:
%
%
\begin{equation}\label{oracle}
\min_{0 \leq k \leq p \wedge n} c^2_k + \sigma^2 \frac{k}{n},
\end{equation}
where
%
%
\begin{equation}\label{oracle: ck}
c^2_k := \min_{ \| \beta\|_0 \leq k} \En[(f_\ii- x_\ii'\beta)^2].
\end{equation}
Note that $c^2_k + \sigma^2 k/n$ is an upper bound on the risk of the
best $k$-sparse
least squares estimator, that is, the best estimator among all least
squares estimators that use $k$ out of $p$ components
of $x_i$ to estimate $f_i$ for $i=1,\ldots,n$. The oracle program
(\ref{oracle})
chooses the optimal value of $k$. Let $s$ be the smallest integer among
these optimal values, and let
%
%
\begin{equation}\label{define beta}
\beta_0 \in\arg\min_{\|\beta\|_0\leq s} \En[(f_\ii- x_\ii'\beta)^2].
\end{equation}
We call $\beta_0$ the oracle target value, $T:= \operatorname{support}(\beta
_0)$ the oracle model, $s:= |T| = \| \beta_0\|_0$ the dimension of the
oracle model, and $x_i'\beta_0$
the oracle approximation to $f_i$. The latter is our
intermediary target, which is equal to the ultimate target $f_i$ up to
the approximation error
\[
r_i := f_i - x_i'\beta_0.
\]
If we knew
$T$, then we could simply use $x_i[T]$
as regressors and estimate $f_i$, for $i=1,\ldots,n$, using the least
squares estimator, achieving the risk of at most
\[
c^2_s + \sigma^2 s/n,
\]
which we call the oracle risk. Because $T$ is not known, we estimate
$T$ using Lasso-type methods and analyze the properties of post-model selection
least squares estimators, accounting for possible model selection mistakes.
\begin{remark}[(The oracle program)]
Note that if argmin is not unique in
the problem (\ref{define beta}), then it suffices to select one of the
values in the set of argmins.
Supplemental article~\cite{BC-PostL1-SA} provides a more detailed
discussion of the oracle problem.
The idea of using oracle problems such as (\ref{oracle}) for
benchmarking the performance follows its previous uses in the
literature (see, e.g.,~\cite{BickelRitovTsybakov2009}, Theorem~6.1, where
an analogous
problem appears in upper bounds on performance of {Lasso}).
\end{remark}
\begin{remark}[(A leading special case)]
When contrasting the performance of Lasso and OLS post-Lasso estimators
in Remarks~\ref{comment: main}  and~\ref{comment: main-goof}  given later, we mention a balanced case where
%
%
\begin{equation}\label{Assump:cs}
c_s^2 \lesssim\sigma^2 s/n,
\end{equation}
which says that the oracle program (\ref{oracle}) is able to balance
the norm of the bias squared to be not much larger than the variance
term $\sigma^2 s/n$. This corresponds to the case where the
approximation error bias does not dominate the estimation error of the
oracle least squares estimator, so that the oracle rate of convergence
simplifies to $\sqrt{s/n}$,
as mentioned in the \hyperref[intr]{Introduction}.
\end{remark}
%
\subsection{Model selectors based on {Lasso}}
Given the large number of regressors $p>n$, some regularization or
covariate selection is needed to obtain consistency.
The {Lasso} estimator~\cite{T1996}, defined as follows, achieves both
tasks by using the $\ell_1$ penalization:
%
%
\begin{equation}\label{Def:{Lasso}main}
\widehat\beta\in\arg\min_{\beta\in\Bbb{R}^p} \widehat Q (\beta
) + \frac{\lambda}{n} \| \beta\|_{1},\qquad\mbox{where } \widehat Q
(\beta) = \En[(y_\ii- x_\ii'\beta)^2],
\end{equation}
and $\lambda$ is the penalty level, the choice of which is described
later. If the solution is not unique, then we pick any solution with
minimum support. The {Lasso} is often used as an estimator, and most
often only as a model selection device, with the model selected by
{Lasso} given by
\[
\widehat T := \operatorname{support}(\widehat\beta).
\]
Moreover, we let $\widehat m := |\widehat T\setminus T|$ denote the
number of components outside $T$ selected by {Lasso} and let $\widehat
f_i = x_i'\widehat\beta, i=1,\ldots,n$, denote the {Lasso} estimate of
$f_i, i=1,\ldots,n$.

Often, additional thresholding is applied to remove regressors with
small estimated coefficients, defining the so-called
``thresholded'' {Lasso} estimator,
%
%
\begin{equation}\label{Def:{t-}{Lasso}}
\hat\beta(t) = ( \hat\beta_j 1\{|\widehat\beta_j|>t\}, j
=1,\ldots,p),
\end{equation}
where $t \geq0$ is the thresholding level. The corresponding selected
model is then
\[
\widehat T(t) := \operatorname{support}(\hat\beta(t)).
\]
Note that, when setting $t=0$, we have $\widehat T(t) = \widehat T$, so
{Lasso} is a special case of thresholded {Lasso}.
%
\subsection{Post-model selection estimators} Given the foregoing, all
of our post-model selection estimators or {OLS} post-{Lasso} estimators
will take the form
%
%
\begin{equation}
\widetilde\beta^t = \arg\min_{\beta\in\RR^p} \widehat Q(\beta)\dvtx
\beta_j = 0 \qquad\mbox{for each } j \in\widehat T^{c}(t).
\end{equation}
That is, given that the model selected a threshold {Lasso} $\widehat T
(t)$, including the {Lasso}'s
model $\widehat T (0)$ as a special case, the post-model selection
estimator applies OLS to the selected model.

Along with the case of $t=0$, we also consider the following choices
for the threshold level:
%
%
\begin{equation}\label{Eq:TL}
\everymath{\displaystyle}
\begin{array}{l@{\quad}l}
\mbox{traditional threshold (t):} & t > \zeta=\max
_{1\leq j\leq p}| \widehat\beta_{j} - \beta_{0j}|, \\
\mbox{fitness-based threshold ({fit}):} & t = t_{\gamma
}:= \max_{t\geq0}\{ t \dvtx \widehat Q(\widetilde\beta^t) -
\widehat Q(\widehat\beta) \leq\gamma\},
\end{array}
\end{equation}
where $\gamma\leq0 $, and $|\gamma|$ is the gain of the in-sample fit
allowed relative to {Lasso}.

As discussed in Section~\ref{Sec:ModelSelection}, the standard
thresholding method is particularly appealing in models in which oracle
coefficients $\beta_0$ are well separated from 0. However, this scheme
may perform poorly in models with oracle coefficients not well
separated from 0 and in nonparametric models. Indeed, even in
parametric models with many small but nonzero true coefficients,
thresholding the estimates too aggressively may result in large
goodness-of-fit losses and, consequently, slow rates of convergence and
even inconsistency for the second-step estimators. This issue directly
motivates our new goodness-of-fit based thresholding method, which
sets as many small coefficient estimates as possible to 0, subject to
maintaining a certain goodness-of-fit level.\looseness=1

Depending on how we select the threshold, we consider three types
of post-model selection estimators:
%
%
\begin{equation}\label{post {Lasso} estimators}
\begin{array}{l@{\quad}l@{\quad}l}
\mbox{{OLS} post-{Lasso}:} & \widetilde\beta^0 & (t =
0), \\[2pt]
\mbox{{OLS} post-{t} {Lasso}:} & \widetilde\beta^t & (t
> \zeta), \\[2pt]
\mbox{{OLS} post-{fit} {Lasso}:} & \widetilde\beta^{t_{\gamma
}} & (t = t_{\gamma}).
\end{array}
\end{equation}
The first estimator is defined by {OLS} applied to the model selected
by {Lasso}, also called Gauss-{Lasso}; the second, by {OLS} applied to
the model selected by the thresholded {Lasso}l and the third, by {OLS}
applied to the model selected
by fitness-thresholded {Lasso}.\looseness=1

The main purpose of this article is to derive the properties of the
post-model selection estimators (\ref{post {Lasso} estimators}).
If model selection works perfectly, which is possible only under rather
special circumstances, then the post-model selection estimators are the
oracle estimators, whose properties are well known. However, of much
more general interest is the case when model selection does not work
perfectly, as occurs for many designs of interest in applications.
%
\subsection{Choice and computation of penalty level for {Lasso}} The
key quantity in the analysis is the gradient of $\widehat Q$ at the
true value,
\[
S = 2 \En[ x_\ii\epsilon_\ii].
\]
This gradient is the effective ``noise'' in the problem that should be
dominated by the regularization. However, we would like to make the
bias as small as possible. This reasoning suggests choosing the
smallest penalty level $\lambda$ possible to dominate
the noise, namely
%
%
\begin{equation}\label{choice of lambda probability}
\lambda\geq c n \|S\|_{\infty}\qquad \mbox{with probability at least } 1-
\alpha,
\end{equation}
where probability $1-\alpha$ needs to be close to 1 and $c>1$.
Therefore, we propose setting
%
%
\begin{equation}\label{choice of lambda}
\lambda= c' \hat\sigma\Lambda(1-\alpha|X)\qquad \mbox{for some fixed }
c' > c > 1,
\end{equation}
where $\Lambda(1-\alpha|X)$ is the $(1-\alpha)$ quantile of $n\|
S/\sigma\|_{\infty}$, and
$\widehat\sigma$ is a possibly data-driven estimate of $\sigma$. Note
that the quantity $\Lambda(1-\alpha|X)$ is independent of $\sigma$ and
can be easily approximated by simulation.
We refer to this choice of $\lambda$ as the data-driven choice,
reflecting the dependence of the choice on the design matrix
$X=[x_1,\ldots,x_n]'$
and a possibly data-driven $\widehat\sigma$.
Note that the proposed (\ref{choice of lambda}) is sharper than
$c'\widehat\sigma2\sqrt{2n \log(p/\alpha)}$ typically used in the
literature.
We impose the following conditions on $\hat\sigma$.
\begin{condition*}[V] \label{conv} The estimated $\hat\sigma$ obeys
\[
\ell\leq\hat\sigma/\sigma\leq u \qquad\mbox{with probability at least }
1- \tau,
\]
where $0 <\ell\leq1$ and $1 \leq u$ and $0\leq\tau< 1$ are
constants possibly dependent on $n$.
\end{condition*}

We can construct a $\widehat\sigma$ that satisfies this condition
under mild assumptions, as follows. First,
set $\widehat\sigma= \hat\sigma_0$, where $\hat\sigma_0$ is an
upper bound on $\sigma$ that is possibly data-driven, for example, the
sample standard deviation of $y_i$.
Second, compute the {Lasso} estimator based on this estimate and set
$\widehat\sigma^2 = \widehat Q(\widehat\beta)$. We demonstrate that
$\hat\sigma$
constructed in this way satisfies Condition \hyperref[conv]{V} and characterize
quantities $u$ and $\ell$ and $\tau$ in the supplemental article
\cite
{BC-PostL1-SA}. We can iterate
on the last step a bounded number of times. We also can use {OLS}
post-{Lasso} for this purpose.
%
\subsection{Choices and computation of thresholding levels}
Our analysis covers a wide range of possible threshold levels. Here,
however, we propose some basic options that give both good
finite-sample and theoretical results. In the traditional
thresholding method, we can set
%
%
\begin{equation}
\label{eq: traditional threshold}
t = \tilde c \lambda/n,
\end{equation}
for some $\tilde c \geq1$. This choice is theoretically motivated by
Section~\ref{Sec:ModelSelection}, which presents the perfect model selection results, where
under some conditions, $\zeta\leq\tilde c \lambda/n$. This choice
also leads to near-oracle performance of the resulting post-model
selection estimator.
Regarding the choice of $\tilde c$, we note that setting $\tilde c=1$
and achieving $\zeta\leq\lambda/n$ is possible based on the results
of Section~\ref{Sec:ModelSelection}
if the empirical Gram matrix is orthogonal and approximation error
$c_s$ vanishes. Thus,
$\tilde c =1$ is the least aggressive traditional thresholding that can
be performed under conditions of Section~\ref{Sec:ModelSelection}. (Also note that $\tilde
c=1$ has performed better than $\tilde c >1$ in our computational experiments.)

Our fitness-based threshold $t_\gamma$ requires specification of the
parameter $\gamma$. The simplest choice delivering
near-oracle performance is $\gamma=0;$ this choice leads to the
sparsest post-model selection estimator
that has the same in-sample fit as {Lasso}. However, we prefer to set
%
%
\begin{equation}\label{DataDrivenGamma}\gamma= \frac{\widehat
Q(\widetilde\beta^{0}) - \widehat Q(\widehat\beta)}{2} < 0,
\end{equation}
where $\widetilde\beta^{0}$ is the {OLS} post-{Lasso}
estimator. The resulting estimator is sparser and produces a better
in-sample fit than {Lasso}.
This choice also results in near-oracle performance and leads to the
best performance in computational experiments.
Note also that for any $\gamma$, we can compute $t_\gamma$ by a binary
search over $t \in\mbox{sort}\{| \hat\beta_j|, j \in\widehat T\}$,
where $\mbox{sort}$ is the sorting operator. This is the case because
the final estimator depends only on the selected support, not on the
specific value of $t$ used. Therefore, because there are at most
$|\widehat T|$ different values of $t$ to be tested, using a binary search,
we can compute $t_\gamma$ exactly by running at most $\lceil\log_2
|\widehat T| \rceil$ OLS problems.
%
%
\subsection{Conditions on the design}
For the analysis of {Lasso}, we use the following restricted eigenvalue
condition on the empirical Gram matrix:
\begin{condition*}[($\boldsymbol{\mathit{RE(\cc)}}$)] \label{conrec} For a given $\cc\geq0$,
\[
\kappa(\cc) : = \min_{\|\delta_{T^c}\|_{1} \leq\cc\| \delta
_{T}\|_{1}, \delta\neq0
} \frac{\sqrt{s}\|\delta\|_{2,n}}{\|\delta_T\|_{1} } > 0.
\]
\end{condition*}

This condition is a variant of the restricted eigenvalue condition
introduced by~\cite{BickelRitovTsybakov2009}, which is
known to be quite general and plausible (see \cite
{BickelRitovTsybakov2009} for related conditions).

For the analysis of post-model selection estimators, we use the
following restricted sparse eigenvalue condition on the empirical Gram matrix:
\begin{condition*}[($\boldsymbol{\mathit{RSE}(m)}$)] For a given $m<n$,
\label{conrsem}
\[
\widetilde\kappa(m)^2 : = \min_{\|\delta_{T^c}\|_{0} \leq m, \delta
\neq0
} \frac{ \|\delta\|_{2,n}^2}{\|\delta\|^2}>0, \qquad\phi(m) : = \max_{\|
\delta_{T^c}\|_{0} \leq m, \delta\neq0
} \frac{ \|\delta\|^2_{2,n}}{\|\delta\|^2}.
\]
\end{condition*}

Condition \hyperref[conrsem]{$\mathit{RSE}(m)$} depends on $T$ and can be viewed as an extension of
the restricted isometry
condition~\cite{CandesTao2007}. Here $m$ denotes the restriction on the
number of nonzero components outside the support $T$. The standard
concept of (unrestricted) $m$-sparse eigenvalues corresponds to the
restricted sparse eigenvalues when $T=\varnothing$ (see, e.g., \cite
{BickelRitovTsybakov2009}). It is convenient to define the following
condition number associated with the empirical Gram
matrix:\looseness=-1
%
%
\begin{equation}\label{Def:mmu}
{\mu(m)} = \frac{\sqrt{\phi(m)}}{\widetilde\kappa(m)}.
\end{equation}\looseness=0
%

The following lemma demonstrates the plausibility of the foregoing
conditions for the case where the values $x_i$, $i=1,\ldots,n$, have
been generated as a realization of the random sample; there are other
primitive conditions as well. In this case, the empirical restricted
sparse eigenvalues are bounded away from 0 and from above, so that (\ref
{Def:mmu}) is bounded from above with high probability. The lemma assumes
as a primitive condition that the sparse eigenvalues of the population
Gram matrix bounded away from zero and from above. The lemma allows for many standard bounded
dictionaries that arise in the nonparametric estimation, for example,
regression splines, orthogonal polynomials, and trigonometric series
(see~\cite{Efromovich1999,vanGeer2000,Wasserman2005,TsybakovBook}).
Similar results are known to hold for standard Gaussian regressors as
well~\cite{ZhangHuang2006}.
\begin{lemma}[(Plausibility of $\boldsymbol{\mathit{RE}}$ and $\boldsymbol{\mathit{RSE}}$)]\label{Lemma:Plausible}
Suppose that $\tilde x_i$, $i = 1,\ldots,n$, are i.i.d. mean-zero vectors,
such that the population Gram matrix $\Ep[\tilde x \tilde x']$ has all
of the diagonal elements equal to~$1$, and
\[
0 < \kappa^2 \leq\min_{1\leq\|\delta\|_0 \leq s\log n} \frac
{\delta
'\Ep[\tilde x \tilde x']\delta}{\|\delta\|} \leq
\max_{1\leq\|\delta\|_0 \leq s\log n} \frac{\delta'\Ep[\tilde x
\tilde
x']\delta}{\|\delta\|} \leq\varphi< \infty.
\]
Define $x_{i}$ as a normalized form of $\tilde x_i$, namely $x_{ij}=
\tilde x_{ij}/(\En[\tilde x_{\ii j}^2])^{1/2}$. Suppose that  \[\max
_{1\leq i\leq n}\|\tilde x_{i}\|_\infty\leq K_n\qquad \mbox{a.s.}\quad  \mbox{and} \quad K^2_ns\log
^2 (n) \log^2 (s\log n) \log(p\vee n) = \mathrm{o}(n \kappa^4/\varphi).\] Then,
for any $m+s\leq s\log n$, the restricted sparse eigenvalues of the
empirical Gram matrix obey the following bounds:
\[
\phi(m) \leq4\varphi, \qquad\widetilde\kappa(m)^2 \geq\kappa^2/4\quad
\mbox{and}\quad {\mu(m)} \leq4\sqrt{\varphi}/\kappa,
\]
with probability approaching 1 as $n \to\infty$.
\end{lemma}
%
\section{\mbox{Results on {Lasso} as an estimator and model~selector}}\label{sec3}
The properties of the post-model selection estimators depend crucially
on both the estimation and model selection properties of {Lasso}. In
this section we
develop the estimation properties of {Lasso} under the data-dependent
penalty level, extending the results of~\cite{BickelRitovTsybakov2009},
and also develop the model selection properties of {Lasso} for
nonparametric models, generalizing the results of~\cite{Lounici2008} to
the nonparametric case.\vadjust{\goodbreak}
%
\subsection{Estimation properties of {Lasso}}
The following theorem describes the main estimation properties of
{Lasso} under the data-driven choice of the penalty level.
\begin{theorem}[(Performance bounds for {Lasso} under data-driven
penalty)]\label{Thm:Nonparametric}
Suppose that Conditions \textup{\hyperref[conm]{M}} and
\hyperref[conrec]{$\mathit{RE}(\cc)$} hold for $\cc=(c+1)/(c-1)$. If $\lambda\geq c n\| S\|
_{\infty
}$, then
\[
\|\widehat\beta- \beta_0\|_{2,n} \leq
\biggl(1 + \frac{1}{c}\biggr) \frac
{\lambda\sqrt{s}}{n \kappa(\cc)} + 2 c_s.
\]
Moreover, suppose that Condition \textup{\hyperref[conv]{V}} holds. Under the data-driven choice
(\ref{choice of lambda}), for $c'\geq c/\ell$, we have $\lambda\geq c
n\| S\|_{\infty} $ with probability at least $1 -\alpha-\tau$, so that
with at least the same probability,
\[
\|\widehat\beta- \beta_0\|_{2,n} \leq(c' + c'/c) \frac{\sqrt{s}}{n
\kappa(\cc)}\sigma u \Lambda(1-\alpha|X) + 2 c_s, \qquad\mbox{where }
\Lambda(1-\alpha|X) \leq\sqrt{2n\log(p/\alpha)}.
\]
If in addition \hyperref[conrec]{$\mathit{RE}(2\cc)$} holds, then
\[
\|\widehat\beta- \beta_0\|_{1} \leq\biggl(\frac
{(1+2\cc)\sqrt{s}}{\kappa(2\cc)} \|\widehat\beta- \beta_0\|
_{2,n} \biggr) \vee\biggl(\biggl(1 + \frac{1}{2\cc}\biggr)\frac{2c}{c-1}\frac
{n}{\lambda} c_s^2 \biggr).
\]
\end{theorem}

This theorem extends the result of~\cite{BickelRitovTsybakov2009} by
allowing for a data-driven
penalty level and deriving the rates in $\ell_1$-norm. These results
may be of independent
interest and are necessary for the subsequent results.
\begin{remark}\label{remark:NormsAndLowerBound} Furthermore, a
performance bound for the estimation of the regression function follows
from the relation
%
%
\begin{equation}\label{Def:NORM_ER}
\bigl| \|\widehat f - f\|_{\Pn,2} - \|\widehat\beta- \beta_0\|
_{2,n}\bigr | \leq c_s,
\end{equation}
where $\widehat f_i = x_i'\widehat\beta$ is the {Lasso} estimate of the
regression function $f$ evaluated at $z_i$.
It is interesting to know some lower bounds on the rate, which follow
from Karush--Kuhn--Tucker conditions
for {Lasso} (see equation (\ref{Eq:Lower{Lasso}}) in the \hyperref[app]{Appendix}):
\[
\|\widehat f - f\|_{\Pn,2} \geq\frac{(1-1/c)\lambda\sqrt{|\widehat
T|}}{2n \sqrt{ \phi(\hat m )}},
\]
where $\widehat m = |\widehat T \setminus T|$. We note that a similar
lower bound was first derived by \cite
{LouniciPontilTsybakovvandeGeer2010} with $\phi(p)$ instead of $\phi
(\hat m)$.
\end{remark}

The preceding theorem and discussion imply the following useful
asymptotic bound on the performance of the estimators.
\begin{corollary}[(Asymptotic bounds on performance of {Lasso})]\label
{Cor:LowerBound}
Under the conditions of Theorem~\ref{Thm:Nonparametric}, if
%
%
\begin{eqnarray}\label{SimpleSC}
\phi(\hat m) &\lesssim&1,\qquad \kappa(\cc) \gtrsim1, \qquad{\mu(\widehat
m)}\lesssim1, \qquad\log(1/\alpha) \lesssim\log p,\nonumber
\\[-8pt]
\\[-8pt]
\alpha&=&\mathrm{o}(1), \qquad u/\ell
\lesssim1 \quad \mbox{and} \quad\tau= \mathrm{o}(1)\nonumber
\end{eqnarray}
%
hold as $n$ grows, then we have
\[
\|\widehat f - f\|_{\Pn,2}
\lesssim_P \sigma\sqrt{\frac{s\log p}{n}} + c_s.
\]
Moreover, if $|\widehat T| \gtrsim_P s$ -- in particular, if
$T\subseteq
\widehat T$ with probability going to $1$ -- then we have
\[
\|\widehat f - f\|_{\Pn,2} \gtrsim_P \sigma\sqrt{\frac{s\log p}{n}}.
\]
\end{corollary}

In Lemma~\ref{Lemma:Plausible} we established fairly general sufficient
conditions for the first three relations in (\ref{SimpleSC}) to hold
with high probability as $n$ grows, when the design
points $z_1,\ldots,z_n$ are generated as a random sample. The
remaining relations are mild conditions on the choice of $\alpha$ and
the estimation of $\sigma$ that are used in the definition of the
data-driven choice (\ref{choice of lambda}) of the penalty-level
$\lambda$.

It follows from the corollary that as long as $\kappa(\cc)$ is bounded
away from 0, {Lasso} with data-driven penalty estimates the regression
function at a near-oracle rate. The second part of the corollary
generalizes to the nonparametric case the lower bound obtained for
{Lasso} by~\cite{LouniciPontilTsybakovvandeGeer2010}. It shows that the
rate cannot be improved in general. We use
the asymptotic rates of convergence to compare the performance of
{Lasso} and the post-model selection estimators.
%
\subsection{Model selection properties of {Lasso}}\label{Sec:ModelSelection}
Our main results do not require that the first-step estimators like
{Lasso} perfectly select the ``true'' oracle model. In fact, we are
specifically interested in the most common cases, where these
estimators do not perfectly select the true model. For these cases, we
prove that post-model selection estimators such as {OLS} post-{Lasso}
achieve near-oracle rates like those of {Lasso}. However,
in some special cases where perfect model selection is possible, these
estimators can achieve the exact oracle rates, and thus can be even
better than {Lasso}. In this section we describe these very special
cases in which perfect model selection is possible.
\begin{theorem}[(Some conditions for perfect model selection in
nonparametric settings)]\label{Lemma:Crack} Suppose that Condition \textup{\hyperref[conm]{M}} holds.

(1) If the coefficients are well separated from 0, that is,
\[
\min_{j\in T} |\beta_{0j}| > \zeta+ t, \qquad\mbox{for some } t\geq
\zeta:= \max_{j=1,\ldots,p} |\widehat\beta_j - \beta_{0j}|,
\]
then the true model is a subset of the selected model,
$
T:=\operatorname{support}(\beta_0) \subseteq\widehat T :=
\operatorname{support}(\widehat\beta).$
Moreover, $T$ can be perfectly selected by applying level $t$
thresholding to $\widehat\beta$, that is, \mbox{$T = \widehat T(t)$}.

(2) In particular, if $\lambda\geq c n\|S\|_{\infty}$ and there is a
constant $U>5\cc$ such that the empirical Gram matrix satisfies $|\En
[x_{\ii j}x_{\ii k}]|\leq1/(Us)$ for all $1\leq j< k\leq p$, then
\[
\zeta\leq\frac{\lambda}{n} \cdot\frac{U + \cc}{U - 5\cc} +
\frac
{\sigma}{\sqrt{n}} \wedge c_s +
\frac{6\cc}{U-5\cc} \frac{c_s}{\sqrt{s}} + \frac{4\cc}{U}\frac
{n}{\lambda} \frac{c_s^2}{s}.
\]
\end{theorem}

These results substantively generalize the parametric results of \cite
{Lounici2008}
on model selection by thresholded {Lasso}. These results cover the more
general nonparametric case and may be of independent
interest. Also note that the stated conditions for perfect model
selection require a strong
assumption on the separation of coefficients of the oracle from 0,
along with near-perfect
orthogonality of the empirical Gram matrix. This is the sense in which
the perfect model selection
is a rather special, nongeneral phenomenon. Finally, we note that it
is possible to perform perfect selection of the oracle model by {Lasso}
without applying any additional thresholding under additional technical
conditions and higher penalty levels \cite
{Bunea2008,Wainright2006,ZhaoYu2006}. In the supplement, we state the
nonparametric extension of the parametric result due to~\cite{Wainright2006}.
%
\subsection{Sparsity properties of {Lasso}}
Here we derive new sharp sparsity bounds for {Lasso}, which may be of
independent interest.We begin with a preliminary sparsity bound for {Lasso}.
\begin{lemma}[(Empirical presparsity for {Lasso})]\label{Lemma:Sparsity{Lasso}}
Suppose that Conditions \textup{\hyperref[conm]{M}} and \hyperref[conrec]{$\mathit{RE}(\cc)$} hold and that $\lambda\geq
c n\|S\|_\infty$, and let $\hat m = |\widehat T \setminus T|$. For
$\cc
= (c+1)/(c-1)$, we have that
\[
\sqrt{\hat m} \leq\sqrt{s}\sqrt{\phi(\hat m)} 2\cc/\kappa(\cc) +
3(\cc+1) \sqrt{\phi(\hat m)} nc_s/\lambda.
\]
\end{lemma}

The foregoing lemma states that {Lasso} achieves the oracle sparsity up
to a factor of $\phi(\hat m)$. Under the conditions (\ref{Assump:cs})
and $\kappa(\cc)\gtrsim1$, the lemma immediately yields the simple
upper bound on the sparsity of the form
%
%
\begin{equation}\label{bad sparsity}
\widehat m \lesssim_P s \phi(n),
\end{equation}
as obtained for examples of~\cite{BickelRitovTsybakov2009} and \cite
{MY2007}. Unfortunately, this bound is sharp only when $\phi(n)$ is
bounded. When $ \phi(n)$ diverges -- for example, when $\phi(n)
\gtrsim
_P \sqrt{\log p}$ in the Gaussian design with $p \geq2n$ by lemma 6 of
\cite{BC-SparseQR-SA} -- the bound is not sharp. However, for this case
we can construct a sharp sparsity bound by combining the preceding
presparsity result with the following sublinearity property of the
restricted sparse eigenvalues.
\begin{lemma}[(Sublinearity of restricted sparse eigenvalues)]
\label{Lemma:SparseEigenvalueIMP}For any integer $k \geq0$ and
constant $\ell\geq1$, we have
$ \phi(\lceil\ell k \rceil) \leq\lceil\ell\rceil\phi(k).$
\end{lemma}

A version of this lemma for (unrestricted) sparse eigenvalues has been
proven by~\cite{BC-SparseQR}. The combination of the preceding two
lemmas gives the following sparsity theorem.
\begin{theorem}[(Sparsity bound for {Lasso} under data-driven
penalty)]\label{Thm:Sparsity}
Suppose that Conditions \textup{\hyperref[conm]{M}} and \hyperref[conrec]{$\mathit{RE}(\cc)$}
hold, and let $\widehat m :=
|\widehat T\setminus T|$. The event $\lambda\geq cn\|S\|_\infty$
implies that
\[
\hat m \leq s \cdot\Bigl[ \min_{m \in\mathcal{M}}\phi(m\wedge n)
\Bigr] \cdot L_n,
\]
where $\mathcal{M}=\{ m \in\mathbb{N}\dvtx
m > s \phi(m\wedge n)\cdot2L_n \}$ and $L_n = [ 2\cc/\kappa(\cc) +
3(\cc+1)nc_s/(\lambda\sqrt{s})]^2$.
\end{theorem}

The main implication of Theorem~\ref{Thm:Sparsity} is that under (\ref
{Assump:cs}), if $ \min_{m \in\mathcal{M}}\phi(m\wedge n)\lesssim1$
and $\lambda\geq cn\|S\|_\infty$ hold with high probability,
which is valid by Lemma~\ref{Lemma:Plausible}
for important designs and by the choice of penalty level (\ref{choice
of lambda}), then, with high probability,
%
%
\begin{equation}\label{eq: oracle sparsity}
\hat m \lesssim s.
\end{equation}
Consequently, for these designs and penalty levels,the sparsity of
{Lasso} is of the same order as that of the oracle, namely $\widehat s
:=|\widehat T| \leq s + \widehat m \lesssim s$, with high probability.
This is because
$ \min_{m \in\mathcal{M}}\phi(m) \ll\phi(n) $ for these designs,
which allows us to sharpen the
previous sparsity bound (\ref{bad sparsity}) considered by \cite
{BickelRitovTsybakov2009} and~\cite{MY2007}. Moreover, our new bound is
comparable to the bounds of~\cite{ZhangHuang2006}
in terms of order of sharpness, but it requires a smaller penalty level
$\lambda$, which also does not depend on the unknown sparse eigenvalues
(as in~\cite{ZhangHuang2006}).
%
%
\section{Performance of post-model selection estimators with a generic
model selector}\label{Sec:PostModel}
Here we present a general result on the performance of a post-model
selection estimator
with a generic model selector.
\begin{theorem}[(Performance of post-model selection estimator with a
generic model selector)]\label{Thm:2StepMain} Suppose that Condition
\textup{\hyperref[conm]{M}} holds, and let $\hat\beta$ be any first-step estimator acting as
the model selector. Denote by
$\widehat T := \operatorname{support}(\widehat\beta)$ the model that
it selects, such
that $|\widehat T| \leq n$. Let $\widetilde\beta$ be the post-model
selection estimator defined by
%
%
\begin{equation}\label{Def:TwoStep}
\widetilde\beta\in\arg\min_{\beta\in\RR^p} \widehat Q(\beta) \dvtx
\beta_j = 0\qquad \mbox{for each } j \in\widehat T^c.
\end{equation}
Let
$ B_n := \widehat Q(\hat\beta) - \widehat Q(\beta_0) \mbox{ and } C_n
:= \widehat Q(\beta_{0\widehat T}) - \widehat Q(\beta_0)$
and  $\hat m=|\widehat T \setminus T| $ be the number of incorrect
regressors selected. Then, if Condition \hyperref[conrsem]{$\mathit{RSE}(\widehat m)$} holds, for
any $\varepsilon>0$, there is a constant $K_\varepsilon$ independent
of $n$ such that with probability at least $1-\varepsilon$, for
$\widetilde f_i = x_i'\widetilde\beta$, we have
\[
\|\widetilde f - f\|_{\Pn,2} \leq
K_{\varepsilon}\sigma\sqrt{\frac{\widehat m \log p + (\widehat m + s
) \log(\mathrm{e}{\mu(\widehat m)})}{n}}+3c_s + \sqrt{(B_n)_+ \wedge
(C_n)_+}.\nonumber\label{Bound:2step}
\]
Furthermore, for any $\varepsilon>0$, there is a constant
$K_\varepsilon$ independent of $n$ such that with probability at least
$1-\varepsilon$,
\begin{eqnarray*}
 B_n &\leq& \| \widehat\beta-\beta_0 \|_{2,n}^2 + \Biggl[
K_{\varepsilon} \sigma\sqrt{\frac{ \widehat m \log p + ( \widehat m +
s ) \log(\mathrm{e}{\mu(\widehat m)})}{n}} + 2 c_s \Biggr] \|\widehat\beta-
\beta_0\|_{2,n} \nonumber\label{Bound:Cn}, \\
\label{Bound:Bn}
 C_n &\leq& 1\{T \not\subseteq\widehat T\} \Biggl( \|\beta
_{0\widehat T^c}\|_{2,n}^2 + \Biggl[K_\varepsilon\sigma\sqrt{\frac{\log
{s\choose\widehat k} +\widehat k\log(\mathrm{e}{\mu(0)})}{n}} + 2c_s \Biggr] \|
\beta_{0\widehat T^c}\|_{2,n} \Biggr).
\end{eqnarray*}
\end{theorem}

Three implications of Theorem~\ref{Thm:2StepMain} are worth noting. First,
the bounds on the prediction norm stated in Theorem~\ref{Thm:2StepMain}
apply to the {OLS} estimator on the components selected by any
first-step estimator $\widehat\beta$,
provided that we can bound both $\| \widehat\beta-\beta_0 \|_{2,n}$,
the rate of convergence of the first-step estimator, and $\widehat m$,
the number of incorrect regressors selected by the model selector.
Second, note that if the selected model
contains the true model, $T \subseteq\widehat T$, then we have
$(B_n)_+\wedge(C_n)_+=C_n =0$. In that case, $B_n$ has no affect on the
rate, and the performance of the second-step estimator is determined by
the sparsity $\widehat m$ of the first-step estimator, which controls
the magnitude of the empirical errors. Otherwise, if the selected model
fails to contain the true model (i.e., $T \not\subseteq\widehat T$),
then the performance of the second-step estimator
is determined by both the sparsity $\widehat m$ and the minimum between
$B_n$ and $C_n$. The quantity $B_n$ measures the
in-sample loss of fit induced by the first-step estimator relative to
the ``true'' parameter value $\beta_0$, and $C_n$ measures the in-sample
loss of fit
induced by truncating the ``true'' parameter $\beta_0$ outside the
selected model $\widehat T$.

The proof of Theorem~\ref{Thm:2StepMain} relies on the sparsity-based
control of the empirical error provided by the following lemma.
\begin{lemma}[(Sparsity-based control of empirical error)]\label{sparse}
Suppose that Condition \textup{\hyperref[conm]{M}} holds.

(1) For any $\varepsilon>0$, there is a constant $K_\varepsilon$
independent of $n$ such that with probability at least $1- \varepsilon
$,
\[
\bigl| \widehat Q(\beta_0+\delta) - \widehat Q(\beta_0) - \| \delta\|
^2_{2,n} \bigr| \leq K_{\varepsilon} \sigma\sqrt{\frac{ m \log p + ( m + s
) \log(\mathrm{e}{\mu(m)})}{n}} \| \delta\|_{2,n}+ 2 c_s \| \delta\|_{2,n},
\]
uniformly for all $\delta\in\Bbb{R}^p$ such that $\|
\delta
_{T^c}\|_0\leq m$, and uniformly over $ m \leq n$.

(2) Furthermore, with at least the same probability,
\[
\bigl| \widehat Q(\beta_{0 \widetilde T}) - \widehat Q(\beta_0) - \|
\beta
_{0 \widetilde T^c} \|^2_{2,n} \bigr| \leq K_{\varepsilon} \sigma\sqrt
{\frac{ \log{s\choose k} + k\log(\mathrm{e}{\mu(0)})}{n}} \| \beta_{0
\widetilde T^c} \|_{2,n}+ 2 c_s \| \beta_{0 \widetilde T^c} \|_{2,n},
\]
uniformly for all $\widetilde T \subset T$ such that
$|T\setminus\widetilde T | = k$, and uniformly over $ k \leq s$.
\end{lemma}

The proof of this lemma in turn relies on the following maximal
inequality, the proof of which involves the use of a
Samorodnitsky--Talagrand type of inequality.
\begin{lemma}[(Maximal inequality for a collection of empirical
processes)]\label{master lemma}
Let $\epsilon_i \sim N(0,\sigma^2)$ be independent for $i=1,\ldots,n$,
and for $m=1,\ldots,n$, define
\[
e_n(m,\eta) := \sigma2\sqrt{2} \Biggl(\sqrt{ \log\pmatrix{p\cr m}} + \sqrt
{(m+s)\log( D {\mu(m)})} + \sqrt{(m+s)\log(1/\eta)}\Biggr)
\]
for any $\eta\in
(0,1)$ and some universal constant $D$. Then,
\[
\sup_{\|\delta_{T^c}\|_0\leq m, \|\delta\|_{2,n} > 0} \biggl| \mathbb
{G}_n\biggl(\frac{\epsilon_i x_i'\delta}{\|\delta\|_{2,n}
}\biggr) \biggr| \leq e_n(m,\eta) \qquad \mbox{ for all } m \leq n,
\]
with probability at least $1-\eta \mathrm{e}^{-s}/(1-1/\mathrm{e})$.
\end{lemma}
%
%
%
\section{Performance of least squares after Lasso-based model
selection}\label{sec5}
In this section we apply our results on post-model selection estimators
to the case where {Lasso} is the first-step estimator. Our previous generic
results allow us to use the sparsity bounds and rate of convergence of
{Lasso} to derive the rate of convergence of post-model selection
estimators in the parametric and nonparametric models.
%
\subsection{Performance of {OLS} post-{Lasso}}
Here we show that the {OLS} post-{Lasso} estimator has good
theoretical performance despite (generally) imperfect selection of the
model by {Lasso}.
\begin{theorem}[(Performance of {OLS} post-{Lasso})]\label
{Cor:2StepNonparametric}
Suppose that Conditions \textup{\hyperref[conm]{M}}, \hyperref[conrec]{$\mathit{RE}(\cc)$}, and \hyperref[conrsem]{$\mathit{RSE}(\widehat m)$} hold,
where $\cc= (c+1)/(c-1)$ and $\widehat m = |\widehat T\setminus T|$.
If $\lambda\geq c n\| S\|_{\infty} $ occurs with probability at least
$1-\alpha$, then for any $\varepsilon> 0$, there is a constant
$K_\varepsilon$ independent of $n$ such that with probability at least
$1- \alpha- \varepsilon$, for $\widetilde f_i = x_i'\widetilde\beta$,
we have
\begin{eqnarray*}
\|\widetilde f - f\|_{\Pn,2} &\leq&
K_{\varepsilon}\sigma\sqrt{\frac{\widehat m \log p + (\widehat m
+ s ) \log(\mathrm{e}{\mu(\widehat m)})}{n}} \\
&&{}+ 3c_s + 1\{T \not\subseteq
\widehat T
\}\sqrt{\frac{\lambda\sqrt{s}}{n \kappa(1)} \biggl(
\frac{(1+c)\lambda\sqrt{s}}{cn \kappa(1)} + 2 c_s\biggr)}.
\end{eqnarray*}

In particular, under Condition \textup{\hyperref[conv]{V}} and the data-driven choice of
$\lambda$ specified in (\ref{choice of lambda}) with $\log(1/\alpha)
\lesssim\log p$, $u/\ell\lesssim1$, for any $\varepsilon>0$ there is
a constant $K_{\varepsilon,\alpha}'$ such that
%
%
\begin{eqnarray}\label{Eq:Post{Lasso}simple}
\|\widetilde f - f\|_{\Pn,2} & \leq&3c_s +
K_{\varepsilon,\alpha}' \sigma\Biggl[\sqrt{ \frac{ \widehat m \log
(p\mathrm{e}{\mu(\widehat m)})}{n} }+ \sqrt{ \frac{ s \log(\mathrm{e}{\mu(\widehat
m)})}{n} }\Biggr]\nonumber
\\[-8pt]
\\[-8pt]
&&{}+ 1\{T \not\subseteq\widehat T \} \Biggl[K_{\varepsilon
,\alpha}'\sigma\sqrt{ \frac{s \log p}{ n}} \frac{1}{ \kappa(1)} + c_s\Biggr]\nonumber
\end{eqnarray}
with probability at least $1-\alpha-\varepsilon-\tau$.\vspace*{-2pt}
\end{theorem}

This theorem provides a performance bound for {OLS} post-{Lasso} as a
function of {Lasso}'s sparsity (characterized by $\widehat m$), rate of
convergence, and model selection ability. For common designs, this
bound implies that {OLS} post-{Lasso} performs at least as well as
{Lasso} and can be strictly better in some cases, and has a smaller
regularization bias. We provide further theoretical comparisons in what
follows, and give computational examples supporting these comparisons
in the supplemental article~\cite{BC-PostL1-SA}. It is also worth
repeating here that performance bounds in other norms of interest
follow immediately by the triangle inequality and by the definition of
$\widetilde\kappa$, as discussed in Remark~\ref{remark:NormsAndLowerBound}.

The following corollary summarizes the performance of {OLS}
post-{Lasso} under commonly used designs.\vspace*{-2pt}
\begin{corollary}[(Asymptotic performance of {OLS} post-{Lasso})]\label
{Cor:Post{Lasso}}
Under the conditions of Theorem~\ref{Cor:2StepNonparametric}, (\ref
{Assump:cs}), and
(\ref{SimpleSC}), as $n$ grows, we have that
\[
\|\widetilde f - f\|_{\Pn,2}\lesssim_P
\cases{\displaystyle\sigma\sqrt{ \frac{ s \log p}{n} } + c_s,&\quad in general,\vspace*{4pt}\cr
 \displaystyle\sigma\sqrt{ \frac{\mathrm{o}(s) \log p}{n} } + \sigma\sqrt{ \frac{ s
}{n} } +c_s,  &\quad if $\widehat m= \mathrm{o}_P(s)$ and $T \subseteq\widehat T$ wp
$\to1$,\vspace*{4pt}\cr
\displaystyle\sigma\sqrt{ s/n } + c_s,& \quad if $T = \widehat T$ wp $\to1$.}
 \]
\end{corollary}
\begin{remark}[(Comparison of the performance of {OLS} post-{Lasso} and
{Lasso})]\label{comment: main} We now compare the upper bounds on the
rates of convergence of {Lasso} and {OLS} post-{Lasso}
under conditions of the corollary. In general, the rates coincide. Of
note, this occurs despite the fact that {Lasso} generally may fail to
correctly select the oracle model $T$ as a subset, that is, $T \not
\subseteq\widehat T$. However, if the oracle model has well-separated
coefficients and conditions and the approximation error does
not dominate the estimation error, then the {OLS} post-{Lasso} rate
improves on the rate of {Lasso}.
Specifically, this occurs if condition (\ref{Assump:cs}) holds and
$\widehat m= \mathrm{o}_P(s)$ and $T \subseteq\widehat T$ wp $\to1$, as under
the conditions of Theorem~\ref{Lemma:Crack} Part 1 or, in the case of
perfect model selection, when $T = \widehat T$ wp $\to1$, as under the
conditions specified by~\cite{Wainright2006}. In such cases, we know
from Corollary~\ref{Cor:LowerBound} that the rates for {Lasso}
are sharp and cannot be faster than $\sigma\sqrt{ s \log p/n}$. Thus
the faster rate of convergence of {OLS} post-{Lasso} over {Lasso} is
strict in such cases.\vspace*{-2pt}
\end{remark}
%
\subsection{Performance of {OLS} post-fit {Lasso}}\vspace*{-2pt}
In what follows we provide performance bounds for {OLS} post-fit
{Lasso} $\widetilde\beta$ defined in equation (\ref{Def:TwoStep}) with
threshold (\ref{Eq:TL}) for the case where the first-step estimator
$\widehat\beta$ is {Lasso}. We let $\widetilde T$ denote the model selected.\vadjust{\goodbreak}
\begin{theorem}[(Performance of {OLS} post-fit {Lasso})]\label
{Cor:Trimmed{Lasso}}
Suppose that Conditions \textup{\hyperref[conm]{M}}, \hyperref[conrec]{$\mathit{RE}(\cc)$}, and \hyperref[conrsem]{$\mathit{RSE}(\widetilde
m)$}
hold, where $\cc= (c+1)/(c-1)$ and $\widetilde m = |\widetilde
T\setminus T|$. If $\lambda\geq c n\| S\|_{\infty} $ occurs with
probability at least $1-\alpha$, then for any $\varepsilon> 0$, there
is a constant $K_\varepsilon$ independent of $n$ such that with
probability at least $1- \alpha- \varepsilon$, for $\widetilde f_i =
x_i'\widetilde\beta$, we have
\begin{eqnarray*}
\|\widetilde f - f\|_{\Pn,2} &\leq&
K_{\varepsilon}\sigma\sqrt{\frac{\widetilde m \log p + (\widetilde m
+ s ) \log(\mathrm{e}{\mu(\widetilde m)})}{n}} \\[-2pt]
&&{}+3c_s +1\{T \not\subseteq
\widetilde T \}\sqrt{\frac{\lambda\sqrt{s}}{n\kappa(1)}\biggl( \frac
{(1+c)\lambda\sqrt{s}}{cn\kappa(1)} + 2c_s\biggr)}.
\end{eqnarray*}

Under Condition \textup{\hyperref[conv]{V}} and the data-driven choice of $\lambda$ specified
in (\ref{choice of lambda}) with $\log(1/\alpha) \lesssim\log p$,
$u/\ell\lesssim1$, for any $\varepsilon>0$ there is a constant
$K_{\varepsilon,\alpha}'$ such that
%
%
\begin{eqnarray}\label{Eq:Post-Good-Simple}
\|\widetilde f - f\|_{\Pn,2} & \leq&3c_s + K_{\varepsilon,\alpha}'
\sigma\Biggl[\sqrt{ \frac{ \widetilde m \log(p\mathrm{e}{\mu(\widetilde m)})}{n} }+
\sqrt{ \frac{ s \log(\mathrm{e}{\mu(\widetilde m)})}{n} } \Biggr] \nonumber
\\[-9pt]
\\[-9pt]
&&{}+ 1\{T \not\subseteq\widetilde T \} \Biggl[ K_{\varepsilon
,\alpha}'\sigma\sqrt{ \frac{s \log p}{ n}} \frac{1}{ \kappa(1)} +
c_s\Biggr ],\nonumber
\end{eqnarray}
with probability at least $1-\alpha-\varepsilon-\tau$.\vspace*{-2pt}
\end{theorem}

This theorem provides a performance bound for {OLS} post-fit {Lasso} as
a function of its sparsity (characterized by $\widetilde m$), {Lasso}'s
rate of convergence, and the model selection ability of the
thresholding scheme. Generally, this bound is as good as the bound for
{OLS} post-{Lasso}, because the {OLS} post-fitness-thresholded {Lasso}
thresholds as much as possible subject to maintaining a certain
goodness of fit. Another appealing feature is that this estimator
determines the thresholding level in a completely data-driven fashion.
Moreover, by construction, the estimated model is sparser than the
{OLS} post-{Lasso} model, which leads to an improved performance of
{OLS} post-fitness-thresholded {Lasso} over {OLS} post-{Lasso} in some
cases. We provide further theoretical comparisons below and
computational examples in the supplemental article~\cite{BC-PostL1-SA}.

The following corollary summarizes the performance of {OLS} post-fit
{Lasso} under commonly used designs.\vspace*{-2pt}
\begin{corollary}[(Asymptotic performance of {OLS} post-fit {Lasso})]
Under the conditions of Theorem~\ref{Cor:Trimmed{Lasso}}, if conditions
in (\ref{Assump:cs}) and (\ref{SimpleSC}) hold, then as $n$ grows, the
{OLS} post-fitness-thresholded {Lasso} satisfies
\[
\|\widetilde f - f\|_{\Pn,2}\lesssim_P
\cases{\displaystyle\sigma\sqrt{ \frac{ s \log p}{n} } + c_s,&\quad in general,\vspace*{2pt}\cr
 \displaystyle\sigma\sqrt{ \frac{\mathrm{o}(s) \log p}{n} } + \sigma\sqrt{ \frac{ s
}{n} } +c_s,  &\quad if $\widetilde m= \mathrm{o}_P(s)$ and $T \subseteq\widetilde T$ wp
$\to1$,\vspace*{2pt}\cr
\displaystyle\sigma\sqrt{ \frac sn } + c_s,& \quad if $T = \widetilde T$ wp $\to1$.}\vadjust{\goodbreak}
 \]
\end{corollary}
\begin{remark}[(Comparison of the performance of {OLS} post-fit {Lasso},
{Lasso}, and {OLS} post-{Lasso})]\label{comment: main-goof}
Under the conditions of the corollary, the {OLS}
post-fitness-thresholded {Lasso} matches the near-oracle rate of
convergence of {Lasso} and {OLS} post-{Lasso}:
$ \sigma\sqrt{s \log p /n} +c_s$. If $\widetilde m= \mathrm{o}_P(s)$ and $T
\subseteq\widetilde T$ wp $\to1$ and (\ref{Assump:cs}) hold, then
{OLS} post-fit {Lasso} strictly improves on {Lasso}'s rate.
That is, if the oracle model has coefficients well separated from 0 and
the approximation error is not dominant, then the improvement is
strict. An interesting question is whether {OLS} post-{fit} {Lasso}
can outperform {OLS} post-{Lasso} in terms of the rates. We cannot rank
these estimators in terms of rates in general; however, this
necessarily occurs when the {Lasso} does not achieve the sufficient
sparsity but the model selection works well, namely when $\widetilde m
= \mathrm{o}_P(\widehat m)$ and $T \subseteq\widetilde T$ wp $\to1$. Finally,
under conditions ensuring perfect model selection -- namely,
the condition of Theorem~\ref{Lemma:Crack} holding for $t= t_{\gamma
}$ -- {OLS} post-fit {Lasso} achieves the oracle performance, \mbox{$\sigma
\sqrt{s /n}+c_s$}.
\end{remark}
%
\subsection{Performance of the {OLS} post-thresholded {Lasso}}
We next consider the traditional thresholding scheme, which truncates
to 0 all components below a set threshold, $t$. This is arguably the
most widely used thresholding scheme in the literature. To state the
result, recall that $\widehat\beta_{tj} = \widehat\beta_j1\{
|\widehat
\beta_j|>t\}$, $\widetilde m := |\widetilde T \setminus T|$, $m_t :=
|\widehat T \setminus\widetilde T|$ and $\gamma_t := \|\widehat\beta
_t - \widehat\beta\|_{2,n}$, where $\widehat\beta$ is the {Lasso} estimator.
\begin{theorem}[(Performance of {OLS} post-t {Lasso})]\label{Cor:Trimmed2}
Suppose that Conditions \textup{\hyperref[conm]{M}}, \hyperref[conrec]{$\mathit{RE}(\cc)$}, and \hyperref[conrsem]{$\mathit{RSE}(\widetilde m)$} hold,
where $\cc= (c+1)/(c-1)$ and $\widetilde m = |\widetilde T\setminus
T|$. If $\lambda\geq c n\| S\|_{\infty} $ occurs with probability at
least $1-\alpha$, then for any $\varepsilon> 0$, there is a constant
$K_\varepsilon$ independent of $n$ such that with probability at least
$1- \alpha- \varepsilon$, for $\widetilde f_i = x_i'\widetilde\beta$,
we have
\begin{eqnarray*}
\|\widetilde f - f\|_{\Pn,2} & \leq&
K_{\varepsilon} \sigma\sqrt{\frac{\widetilde m \log p + (
\widetilde
m + s ) \log(\mathrm{e}{\mu(\widetilde m)})}{n}} + 3c_s \\
&&{}+1\{T \not\subseteq\widetilde T \} \biggl(\gamma_t + \frac{1+c}{c}\frac
{\lambda\sqrt{s}}{n\kappa(\cc)}+2c_s\biggr)
+ 1\{T \not\subseteq\widetilde T \}\\
&&{}\times\sqrt{ \Biggl[ K_{\varepsilon}
\sigma\sqrt{\frac{ \widetilde m \log p + ( \widetilde m + s ) \log
(\mathrm{e}{\mu(\widetilde m)})}{n}} + 2 c_s \Biggr] \biggl( \gamma_t + \frac
{1+c}{c}\frac{\lambda\sqrt{s}}{n\kappa(\cc)}+2c_s\biggr)},
\end{eqnarray*}
where 
$\gamma_t \leq t\sqrt{\phi(m_t) m_t}$. Under Condition \textup{\hyperref[conv]{V}} and the
data-driven choice of $\lambda$ specified in (\ref{choice of lambda})
for $\log(1/\alpha) \lesssim\log p$, $u/\ell\lesssim1$, for any
$\varepsilon>0$, there is a constant $K_{\varepsilon,\alpha}'$ such
that with probability at least $1-\alpha-\varepsilon- \tau$,
\begin{eqnarray*}
\|\widetilde f - f\|_{\Pn,2} &\leq&3c_s+
K_{\varepsilon,\alpha}' \Biggl[ \sigma\sqrt{\frac{\widetilde m \log
(p\mathrm{e}{\mu(\widetilde m)})}{n}} + \sigma\sqrt{\frac{ s\log(\mathrm{e}{\mu
(\widetilde m)})}{n}} \Biggr]\\
&&{}+ 1\{T \not
\subseteq\widetilde T \} \Biggl[ \gamma_t + K_{\varepsilon,\alpha
}'\sigma\sqrt{\frac{s\log p}{n}}\frac{1}{\kappa(\cc)} + 4c_s \Biggr].
\end{eqnarray*}
\end{theorem}

This theorem provides a performance bound for {OLS} post-thresholded
{Lasso} as a function of (1) its sparsity, characterized by $\widetilde
m$, and improvements in sparsity over {Lasso}, characterized by $m_t$;
(2) {Lasso}'s rate of convergence; (3) the thresholding level $t$ and
resulting goodness-of-fit loss, $\gamma_t$, relative to {Lasso} induced
by thresholding; and (4) the model selection ability of the
thresholding scheme. Generally, this bound may be worse than the bound
for {Lasso}, because the {OLS} post-thresholded {Lasso} potentially
uses too much thresholding, resulting in large goodness-of-fit losses,
$\gamma_t$. We provide further theoretical comparisons below and
computational examples in Section 4 of the supplemental article \cite
{BC-PostL1-SA}.


\begin{remark}[(Comparison of the performance of {OLS} post-thresholded
{Lasso}, {Lasso}, and {OLS} post-{Lasso})]\label{comment: main-goofa}
In this work, we also assume conditions in (\ref{Assump:cs}) and (\ref
{SimpleSC}) presented in the foregoing formal comparisons. Under these
conditions, {OLS} post-thresholded {Lasso} obeys the bound
%
%
\begin{equation}\label{Eq:Trad-SimpleSimple}
\|\widetilde f - f\|_{\Pn,2}\lesssim_P \sigma\sqrt{ \frac{
\widetilde
m \log p}{n} }+ \sigma\sqrt{ \frac{ s }{n} } + c_s + 1\{T \not
\subseteq
\widetilde T \} \Biggl( \gamma_t \vee\sigma\sqrt{\frac{s \log p}{ n}}
\Biggr).
\end{equation}
In this case, we have $\widetilde m \vee m_t \leq s+\widehat m \lesssim
_P s$ by Theorem~\ref{Thm:Sparsity}. In general, the foregoing rate
cannot improve on {Lasso}'s rate of convergence given in Lemma \ref
{Thm:Nonparametric}.

As expected, the choice of $t$, which controls $\gamma_t$ via the bound
$\gamma_t \leq t\sqrt{\phi(m_t) m_t}$, can have a significant effect on
the performance bounds. If
%
%
\begin{equation}\label{best t}
t \lesssim\sigma\sqrt{\frac{\log p}{n}} \qquad\mbox{then } \|
\widetilde f - f\|_{\Pn,2}\lesssim_P \sigma\sqrt{ \frac{ s \log p}{n}
} + c_s.
\end{equation}
The choice (\ref{best t}), suggested by~\cite{Lounici2008} and Theorem
\ref{Thm:Sparsity}, is theoretically sound, because it guarantees that {OLS}
post-thresholded {Lasso} achieves the near-oracle rate of {Lasso}. Note
that to implement the choice (\ref{best t}) in practice, we suggest
setting $t = \lambda/n$, given that the separation of the coefficients
from 0 is unknown in practice. Note that using a much larger $t$ can
lead to inferior rates of convergence.

Furthermore, there is a special class of models -- a neighborhood of
parametric models with well-separated coefficients -- for which
improvements in the rate of convergence of {Lasso} are possible.
Specifically, if $\widetilde m= \mathrm{o}_P(s)$ and $T \subseteq\widetilde T$
wp $\to1$, then {OLS} post-thresholded {Lasso} strictly improves on
the {Lasso}'s rate. Furthermore, if $\widetilde m = \mathrm{o}_P(\widehat m)$
and $T \subseteq\widetilde T$ wp $\to1$, then {OLS} post-thresholded
{Lasso} also outperforms {OLS} post-{Lasso}:
\[
\|\widetilde f - f\|_{\Pn,2} \lesssim_P \sigma\sqrt{ \frac{
\mathrm{o}(\widehat m) \log p}{n} } + \sigma\sqrt{ \frac{ s }{n} } + c_s.
\]
Finally, with the conditions of Theorem~\ref{Lemma:Crack} holding for
given $t$, {OLS} post-thresholded {Lasso} achieves oracle performance,
$\|\widetilde f - f\|_{\Pn,2} \lesssim_P \sigma\sqrt{s /n} + c_s$.
\end{remark}
\begin{appendix}\label{app}
\setcounter{equation}{0}
\section*{Appendix: Proofs}
\setcounter{subsection}{0}
\subsection{\texorpdfstring{Proofs for Section \protect\ref{sec3}}{Proofs for Section 3}}
\begin{pf*}{Proof of Theorem~\ref{Thm:Nonparametric}}
The bound in $\|\cdot\|_{2,n}$ norm follows by the same steps specified
by~\cite{BickelRitovTsybakov2009}, and thus we defer the derivation to
the supplement.

Under the data-driven choice (\ref{choice of lambda}) of $\lambda$ and
Condition \textup{\hyperref[conv]{V}}, we have $c'\widehat\sigma\geq c \sigma$ with probability
at least $1-\tau$, because $c'\geq c/\ell$. Moreover, with the same
probability, we also have $\lambda\leq c'u\sigma\Lambda(1-\alpha|X)$.
The result follows by invoking the $\|\cdot\|_{2,n}$ bound.

The bound in $\|\cdot\|_1$ is proven as follows.
First, assume that $\|
\delta_{T^c}\|_1 \leq2\cc\| \delta_T\|_1.$ In this
case, by the definition of the restricted eigenvalue, we have
$ \|\delta\|_1 \leq(1+2\cc) \|\delta_T\|_1 \leq(1+2\cc)\sqrt{s}\|
\delta\|_{2,n}/\kappa(2\cc), $
and the result follows by applying the first bound to
$\|\delta\|_{2,n}$ because $\cc> 1$. On the other hand, consider the
case where $\|\delta_{T^c}\|_1 > 2\cc\|\delta_T\|_1$. Here the
relation
\[
-\frac{\lambda}{cn} (\| \delta_T\|_{1} + \| \delta_{T^c}\|_{1} ) +
\|
\delta\|_{2,n}^2 - 2 c_s \| \delta\|_{2,n} \leq\frac{\lambda}{n} (
\|
\delta_T\|_{1} - \|\delta_{T^c}\|_{1}),
\]
which is established in (2.3) in the supplemental article~\cite{BC-PostL1-SA},
implies that $\|\delta\|_{2,n} \leq2c_s$ and also
\[
\|\delta_{T^c}\|_1 \leq\cc\|\delta_T\|_1 + \frac{c}{c-1} \frac
{n}{\lambda}\|\delta\|_{2,n}(2c_s - \|\delta\|_{2,n})
\leq\|\delta_T\|_1 + \frac{c}{c-1}\frac{n}{\lambda} c_s^2 \leq
\frac
{1}{2} \| \delta_{T^c}\|_1 + \frac{c}{c-1}\frac{n}{\lambda} c_s^2.
\]
Thus,
\[
\| \delta\|_1 \leq\biggl(1 + \frac{1}{2\cc}\biggr) \| \delta_{T^c}\|_1 \leq\biggl(1 +
\frac{1}{2\cc}\biggr)\frac{2c}{c-1}\frac{n}{\lambda} c_s^2.
\]
The result follows by taking the maximum of the bounds on each case and
invoking the bound on $\|\delta\|_{2,n}$.
\end{pf*}
\begin{pf*}{Proof of Theorem~\ref{Lemma:Crack}}
Part (1) follows immediately from the assumptions. To show part (2),
let $\delta:=\widehat\beta-\beta_0$, and proceed in two steps:

Step 1. By the first-order optimality conditions of $\widehat\beta$
and the assumption on $\lambda$,
\begin{eqnarray*}
\|\En[x_\ii x_\ii'\delta]\|_{\infty} & \leq&\|\En
[x_\ii(y_\ii-x_\ii'\widehat\beta)]\|_{\infty} + \|S/2\|_{\infty}
+ \|\En
[x_\ii r_\ii]\|_{\infty} \\
& \leq&\frac{\lambda}{2n} +
\frac{\lambda}{2cn} + \min\biggl\{\frac{\sigma}{\sqrt{n}}, c_s \biggr\},
\end{eqnarray*}
because $\|\En[x_\ii r_\ii]\|_{\infty} \leq\min\{\frac{\sigma
}{\sqrt{n}}, c_s \}$ by step 2 below.

Next, let $e_j$ denote the $j$th canonical direction. Thus, for every
$j=1,\ldots,p$, we have
\begin{eqnarray*}
| \En[e_j'x_\ii x_\ii'\delta] - \delta_j |
&=&|\En[e_j'(x_\ii x_\ii'-I)\delta]| \\
& \leq&\max_{1\leq j,k\leq
p}|(\En
[x_\ii x_\ii'-I])_{jk}| \|\delta\|_{1}\\
& \leq&\|\delta\|_1/[Us].
\end{eqnarray*}
Then, combining the two bounds above and using the triangle inequality,
we have
\[
\|\delta\|_\infty\leq\|\En[x_\ii x_\ii'\delta]\|_{\infty} + \|
\En
[x_\ii x_\ii'\delta]-\delta\|_{\infty} \leq\biggl(1+\frac{1}{c}\biggr)\frac
{\lambda
}{2n} + \min\biggl\{\frac{\sigma}{\sqrt{n}}, c_s \biggr\} +
\frac{\|\delta\|_1}{Us}.
\]
The result follows by Theorem~\ref{Thm:Nonparametric} to bound $\|
\delta
\|_1$ and the
arguments of~\cite{BickelRitovTsybakov2009} and~\cite{Lounici2008} to
show that the bound on the correlations imply that for any $C>0$,
\[
\kappa(C) \geq\sqrt{ 1 - s(1+2C)\|\En[x_\ii x_\ii'-I]\|_\infty},
\]
so that
$\kappa(\cc) \geq\sqrt{1-[(1+2\cc)/U]}$ and $\kappa(2\cc) \geq
\sqrt
{1-[(1+4\cc)/U]}$ under this particular design.

Step 2. In this step, we show that $ \| \En[ x_\ii r_\ii]\|_\infty
\leq\min\{\frac{\sigma}{\sqrt{n}}, c_s \}.$
First, note that for every $j=1,\ldots, p$, we have $|\En[x_{\ii
j}r_\ii
]|\leq\sqrt{\En[x_{\ii j}^2]\En[r_\ii^2]}=c_s$.
Next, by the definition of $\beta_0$ in (\ref{oracle}), for $j \in T$,
we have
$\En[x_{\ii j}(f_\ii-x_\ii'\beta_0)] = \En[x_{\ii j}r_\ii] = 0 $,
because $\beta_0$ is a minimizer over the support of $\beta_0$. For $j
\in T^c$, we have that for any $t \in\RR$,
\[
\En[(f_\ii- x_\ii'\beta_0)^2] + \sigma^2 \frac{s}{n} \leq\En
[(f_\ii
- x_\ii'\beta_0-tx_{\ii j})^2] + \sigma^2 \frac{s+1}{n}.
\]
Therefore, for any $t\in\RR$, we have
\[
-\sigma^2/n \leq\En[(f_\ii- x_\ii'\beta_0-tx_{\ii j})^2] - \En
[(f_\ii- x_\ii'\beta_0)^2] = -2t\En[x_{\ii j}(f_\ii- x_\ii'\beta
_0)]+t^2\En[x_{\ii j}^2].
\]
Taking the minimum over $t$ on the right-hand side at $t^* = \En
[x_{\ii j}(f_\ii- x_\ii'\beta_0)]$, we obtain
$ -\sigma^2/n \leq- (\En[x_{\ii j}(f_\ii- x_\ii'\beta_0)])^2$ or,
equivalently, $|\En[x_{\ii j}(f_\ii- x_\ii'\beta_0)]|\leq\sigma
/\sqrt{n}$.
\end{pf*}
\begin{pf*}{Proof of Lemma~\ref{Lemma:Sparsity{Lasso}}} Let
$\widehat T = \operatorname{support}(\hat\beta)$ and $\hat m =
|\widehat T\setminus T|$.
From the optimality conditions, we have that $|2\En[ x_{\ii j}(y_\ii
-x_\ii'\hat\beta)]| = \lambda/n \mbox{ for all } j \in\widehat T.
$ Therefore, for $R=(r_1,\ldots,r_n)'$, we have
\begin{eqnarray*}
\sqrt{|\widehat T|}\lambda& \leq& 2\bigl\| \bigl(X'(Y - X\hat\beta
)\bigr)_{\widehat T}\bigr \| \\[-2pt]
& \leq& 2\bigl\| \bigl(X'(Y - R - X \beta_0)\bigr)_{\widehat T} \bigr\| + 2\bigl\| \bigl(X'(R +
X\beta_0- X\hat\beta)\bigr)_{\widehat T} \bigr\| \\[-2pt]
& \leq& \sqrt{|\widehat T|}\cdot n\|S\|_{\infty} + 2n \sqrt{ \phi
(\hat m )} \bigl(\En[ (x_\ii'\hat\beta- f_\ii)^2]\bigr)^{1/2},
\end{eqnarray*}
using the definition of $\phi(\hat m)$ and the Holder inequality,
\begin{eqnarray*}
\bigl\| \bigl(X'(R + X\beta_0- X\hat\beta)\bigr)_{\widehat T}\bigr
 \|&
\leq&\sup_{\|\alpha_{T^c}\|_0\leq\hat m, \|\alpha\|\leq1}|
\alpha' X'(R + X\beta_0- X\hat\beta)| \\[-2pt]
&\leq&
\sup_{\|\alpha_{T^c}\|_0\leq\hat m, \|\alpha\|\leq1}\| \alpha'X'\|
\|R
+ X\beta_0- X\hat\beta\| \\[-2pt]
&= & \sup_{\|\alpha_{T^c}\|_0\leq\hat
m, \|\alpha\|\leq1}\sqrt{| \alpha'X'X\alpha|}\|R + X\beta_0-
X\hat
\beta\|\\[-2pt]
& = &
n\sqrt{\phi(\hat m)}\bigl(\En[ (x_\ii'\hat\beta- f_\ii)^2]\bigr)^{1/2}.
\end{eqnarray*}
Because $\lambda/c \geq n\|S\|_\infty$, we have
%
%
\begin{equation}\label{Eq:Lower{Lasso}} (1-1/c)\sqrt{|\widehat
T|}\lambda\leq2n \sqrt{ \phi(\hat m )} \bigl(\En[ (x_\ii'\hat\beta-
f_\ii
)^2]\bigr)^{1/2}.
\end{equation}
Moreover, because $\widehat m \leq|\widehat T|$, and by Theorem \ref
{Thm:Nonparametric} and Remark~\ref{remark:NormsAndLowerBound}, $(\En[
(x_\ii'\hat\beta- f_\ii)^2])^{1/2} \leq\|\widehat\beta-\beta_0\|
_{2,n}+c_s \leq(1 + \frac{1}{c}) \frac{\lambda\sqrt{s}}{n \kappa
(\cc
)} + 3c_s$, we have
\[
(1-1/c)\sqrt{\hat m} \leq2\sqrt{\phi(\hat m)}(1+1/c)\sqrt
{s}/\kappa(\cc
) + 6 \sqrt{\phi(\hat m)} nc_s/\lambda.
\]

The result follows by noting that $(1-1/c) = 2/(\cc+1)$ by definition
of $\cc$.
\end{pf*}
\begin{pf*}{Proof of Theorem~\ref{Thm:Sparsity}}
By Lemma~\ref{Lemma:Sparsity{Lasso}}, $ \sqrt{\hat m} \leq\sqrt{ \phi(\hat m)} \cdot2\cc\sqrt
{s}/\kappa(\cc
) + 3(\cc+1) \sqrt{\phi(\hat m)} \cdot nc_s/\lambda,$
which, by letting $L_n = ( \frac{2\cc}{\kappa(\cc)} + 3(\cc
+1)\frac
{nc_s}{\lambda\sqrt{s}} )^2$, can be rewritten as
%
%
\begin{equation}\label{Eq:Sparsity}\hat m \leq s \cdot\phi(\hat m) L_n.
\end{equation}
Note that $\widehat m \leq n$ by optimality conditions. Consider any $M
\in\mathcal{M}$, and suppose that $\widehat m > M$. Therefore, by
Lemma~\ref{Lemma:SparseEigenvalueIMP} on the sublinearity of restricted
sparse eigenvalues,
\[
\hat m \leq s \cdot\biggl\lceil\frac{\hat m}{M} \biggr\rceil\phi(M) L_n.
\]
Thus, because $\lceil k \rceil< 2k$ for any $k\geq1$, we have
$ M < s \cdot2\phi(M) L_n$, which violates the condition of $M \in
\mathcal{M}$. Therefore, we must have $\widehat m \leq M$.
In turn, applying (\ref{Eq:Sparsity}) once more with $\widehat m \leq
(M\wedge n)$, we obtain
$ \hat m \leq s \cdot\phi(M\wedge n) L_n.$
The result follows by minimizing the bound over $M \in\mathcal{M}$.
\end{pf*}
%
\subsection{\texorpdfstring{Proofs for Section \protect\ref{Sec:PostModel}}{Proofs for Section 4}}
\begin{pf*}{Proof of Theorem~\ref{Thm:2StepMain}}
Let $\widetilde\delta:= \widetilde\beta- \beta_0$. By the
definition of the second-step estimator, it follows that $\widehat
Q(\widetilde\beta) \leq\widehat Q(\widehat\beta)$ and $\widehat
Q(\widetilde\beta) \leq\widehat Q(\beta_{0\widehat T})$. Thus,
\[
\widehat Q (\widetilde\beta) - \widehat Q(\beta_0) \leq\bigl( \widehat Q
(\widehat\beta) - \widehat Q(\beta_0) \bigr) \wedge\bigl( \widehat Q(\beta
_{0\widehat T}) - \widehat Q(\beta_0) \bigr) \leq B_n \wedge C_n.
\]
By Lemma~\ref{sparse} part (1), for any $\varepsilon>0$ there exists a
constant $K_{\varepsilon}$ such that with probability at least
$1-\varepsilon$,
$| \widehat Q(\widetilde\beta) - \widehat Q(\beta_0) - \|\widetilde
\delta\|^2_{2,n} | \leq A_{\varepsilon,n}\|\widetilde\delta\|_{2,n} +
2c_s\|\widetilde\delta\|_{2,n}$,
where
\[
A_{\varepsilon,n} := K_{\varepsilon} \sigma
\sqrt{\bigl(\widehat m \log p + ( \widehat m + s )
\log(\mathrm{e}{\mu(\widehat m)})\bigr)/n}.
\]
Combining these relations, we obtain the inequality
$
\|\widetilde\delta\|_{2,n}^2 - A_{\varepsilon,n} \|\widetilde\delta
\|
_{2,n} - 2c_s\|\widetilde\delta\|_{2,n}\leq B_n \wedge C_n.
$
Solving this, we obtain the stated inequality,
$ \|\widetilde\delta\|_{2,n} \!\leq\! A_{\varepsilon,n}\! + \!2c_s\!+\! \sqrt
{(B_n)_+ \wedge(C_n)_+}.$
Finally, the bound on $B_n$ follows from Lemma~\ref{sparse} part (1).
The bound on $C_n$ follows from Lemma~\ref{sparse} part~(2).
\end{pf*}
\begin{pf*}{Proof of Lemma~\ref{sparse}} The proof of part (1)
follows from the relation
\[
\bigl| \widehat Q(\beta_0 + \delta) - \widehat Q(\beta_0) - \| \delta
\|
^2_{2,n} \bigr| = | 2 \En[ \epsilon_\ii x_\ii' \delta] + 2 \En[r_\ii
x_\ii'
\delta]|,
\]
and then bounding $|2 \En[r_\ii x_\ii'\delta]|$ by $2 c_s \|
\delta\|
_{2,n}$ using the Cauchy--Schwarz inequality, applying Lemma \ref
{master lemma} on sparse control of noise to $|2 \En[ \epsilon_\ii
x_\ii' \delta]|$, where we bound ${p\choose m}$ by $p^m$ and set
$K_\varepsilon= 6\sqrt{2}\log^{1/2} \max\{e, D, 1/(\mathrm{e}^s\varepsilon
[1-1/\mathrm{e}])\}$. The proof part (2) also follows from Lemma~\ref{master
lemma}, but applying it with $s=0$, $p=s$ (because only the components
in $T$ are modified), $m = k$, and noting that we can take ${\mu(m)}$
with $m=0$.
\end{pf*}
\begin{pf*}{Proof of Lemma~\ref{master lemma}} We divide the proof
into steps.

Step 0. Note that we can restrict the supremum over $\|\delta\|=1$
because the function is homogenous of degree 0.

Step 1. For each nonnegative integer $ m \leq n$ and each set
$\widetilde T \subset\{1,\ldots,p\}$, with $|\widetilde T\setminus
T|\leq m$, define the class of functions
%
%
\begin{equation}\label{Def:FF}
\mathcal{G}_{\widetilde T} = \{ \epsilon_i x_i'\delta/ \|\delta\|
_{2,n} \dvtx \operatorname{support}(\delta) \subseteq\widetilde T, \|\delta\|
=1\}.
\end{equation}
Also define $\mathcal{F}_{m} = \{ \mathcal{G}_{\widetilde T} \dvtx
\widetilde T \subset\{1,\ldots,p\}\dvtx | \widetilde T\setminus T|\leq m\}.
$
It follows that
%
%
\begin{equation}\label{Eq:FirstIneq}
P\Bigl( \sup_{f\in\mathcal{F}_m}|\Gn(f)| \geq e_n(m,\eta) \Bigr) \leq\pmatrix{p\cr m} \max_{|\widetilde T\setminus T|\leq m} P\Bigl( \sup_{f\in\mathcal
{G}_{\widetilde T}}|\Gn(f)| \geq e_n(m,\eta) \Bigr).
\end{equation}

We apply the Samorodnitsky--Talagrand inequality (Proposition A.2.7 of van
der Vaart and Wellner~\cite{vdV-W}) to bound the right-hand side
of (\ref{Eq:FirstIneq}). Let
\[
\rho(f,g) := \sqrt{ \mathrm{E}[\Gn(f) - \Gn(g)]^2}= \sqrt{\mathrm{E}\En[(f-g)^2]}
\]
for $f, g \in\mathcal{G}_{\widetilde T}$. By step 2 below, the
covering number of $\mathcal{G}_{\widetilde T}$ with respect to $\rho$
obeys
%
%
\begin{equation}\label{Bound Covering Number}
N(\varepsilon,\mathcal
{G}_{\widetilde T},\rho) \leq\bigl(6 \sigma
{\mu(m)}/\varepsilon\bigr)^{m+s}\qquad \mbox{for each } 0 < \varepsilon\leq
\sigma,
\end{equation}
and
$\sigma^2(\mathcal{G}_{\widetilde T}) := \max_{f \in\mathcal
{G}_{\widetilde T}} \mathrm{E}[\Gn(f)]^2= \sigma^2$. Then, by the
Samorodnitsky--Talagrand inequality,
%
%
\begin{equation}\label{TalagrandIneq}
P\Bigl( \sup_{f\in\mathcal
{G}_{\widetilde T}}|\Gn(f)| \geq e_n(m,\eta)\Bigr ) \leq\biggl( \frac{D
\sigma{\mu(m)} e_n(m,\eta)}{\sqrt{m+s}\sigma^2} \biggr)^{m+s} \bar{\Phi
}\bigl(e_n(m,\eta)/\sigma\bigr)
\end{equation}
for some universal constant $D\geq1$, where $\bar{\Phi}=1-\Phi$ and
$\Phi$ is the cumulative probability distribution function for a
standardized Gaussian random variable.
For $e_n(m,\eta)$ defined in the statement of the theorem, it follows that
$ P( \sup_{f\in\mathcal{G}_{\widetilde T}}|\Gn(f)| \geq e_n(m,\eta
) )
\leq\eta e^{-m-s}/{p\choose m}$ by simple substitution into (\ref
{TalagrandIneq}).
Then,
\begin{eqnarray*}
P \Bigl( \sup_{f \in\mathcal{F}_m} |
\mathbb{G}_n(f)| > e_n(m,\eta), \exists m \leq n \Bigr) & \leq& \sum
_{m=0}^n P \Bigl( \sup_{f \in\mathcal{F}_m} | \mathbb{G}_n(f)| >
e_n(m,\eta) \Bigr) \\
& \leq& \sum_{m=0}^n \eta \mathrm{e}^{-m-s} \leq\eta \mathrm{e}^{-s}/(1-1/\mathrm{e}),
\end{eqnarray*}
which proves the claim.

Step 2. This step establishes (\ref{Bound Covering Number}). For $t
\in
\Bbb{R}^p$ and $\widetilde t \in\Bbb{R}^p$,
consider any two functions
\[
\epsilon_i \frac{(x_i't)}{ \|t\|_{2,n}} \mbox{ and } \epsilon_i
\frac
{(x_i'\widetilde t)}{ \|\widetilde t\|_{2,n}} \mbox{ in } \mathcal
{G}_{\widetilde T}, \mbox{ for a given } \widetilde T \subset\{
1,\ldots
,p\}\dvtx |\widetilde T\setminus T|\leq m.
\]
We have that
\[
\sqrt{ \mathrm{E}\En\biggl[ \epsilon_\ii^2 \biggl( \frac{(x_\ii't)}{ \|t\|_{2,n}} -
\frac
{(x_\ii'\widetilde t)}{\|\widetilde t\|_{2,n}} \biggr)^2\biggr]} \leq\sqrt{ \mathrm{E}\En\biggl[
\epsilon_\ii^2 \frac{( x_\ii'(t-\widetilde t))^2}{ \|t\|^2_{2,n}}\biggr]} +
\sqrt{ \mathrm{E}\En\biggl[ \epsilon_\ii^2 \biggl( \frac{(x_\ii'\widetilde t)}{\|t\|_{2,n}}
-\frac{(x_\ii'\widetilde t)}{\|\widetilde t\|_{2,n}}\biggr)^2\biggr]}.
\]

By definition of $\mathcal{G}_{\widetilde T}$ in (\ref{Def:FF}),
$\operatorname{support}
(t) \subseteq\widetilde T$ and $\operatorname{support}(\widetilde t)
\subseteq
\widetilde T$, so that $\operatorname{support}(t-\widetilde t)
\subseteq\widetilde T$,
$|\widetilde T\setminus T|\leq m$, and $\|t\| = 1$ by (\ref{Def:FF}).
Thus, by the definition of \hyperref[conrsem]{$\mathit{RSE}(m)$},
\begin{eqnarray*}
\mathrm{E}\En\biggl[ \epsilon_\ii^2 \frac{( x_\ii'(t-\widetilde t))^2}{ \|t\|
^2_{2,n}}\biggr] &\leq&\sigma^2 \phi(m) \|t-\widetilde t\|^2/ \widetilde
\kappa(m)^2, \quad\mbox{and}\\
\mathrm{E}\En\biggl[ \epsilon_\ii^2 \biggl( \frac{(x_\ii'\widetilde t)}{\|t\|_{2,n}}
-\frac{(x_\ii'\widetilde t)}{\|\widetilde t\|_{2,n}}\biggr)^2\biggr]
&=& \mathrm{E}\En\biggl[
\epsilon_\ii^2 \frac{(x_\ii'\widetilde t)^2}{\|\widetilde t\|^2_{2,n}}
\biggl(\frac{\|\widetilde t\|_{2,n}- \|t\|_{2,n} }{\|t\|_{2,n}}\biggr )^2\biggr] \\
&=&  \sigma^2\biggl(\frac{\|\widetilde t\|_{2,n}- \|t\|_{2,n} }{\|t\|_{2,n}}
\biggr)^2\\
&\leq&\sigma^2 \|\widetilde t-t\|_{2,n}^2/\|t\|^2_{2,n}
\leq\sigma
^2 \phi(m)\|\widetilde t-t\|^2/\widetilde\kappa(m)^2,
\end{eqnarray*}
so that
\[
\sqrt{ \mathrm{E}\En\biggl[ \epsilon_\ii^2 \biggl( \frac{(x_\ii't)}{ \|t\|_{2,n}} -
\frac
{(x_\ii'\widetilde t)}{\|\widetilde t\|_{2,n}}\biggr )^2\biggr]} \leq2\sigma\|
t-\widetilde t\|\sqrt{\phi(m)}/\widetilde\kappa(m) = 2 \sigma{\mu(m)}
\|t-\widetilde t\|.
\]
Then the bound (\ref{Bound Covering Number}) follows from the bound of
\cite{vdV-W}, page 94, $N(\varepsilon,\mathcal{G}_{\widetilde T},\rho
) \leq N(\varepsilon/R,\break B(0,1),\|\cdot\|)\leq(3R/\varepsilon
)^{m+s} $, with $R = 2\sigma{\mu(m)}$ for any $\varepsilon\leq
\sigma$.
\end{pf*}
%
%
\subsection{\texorpdfstring{Proofs for Section \protect\ref{sec5}}{Proofs for Section 5}}
\begin{pf*}{Proof of Theorem~\ref{Cor:2StepNonparametric}} First,
note that if $T \subseteq\widehat T$, we then have $C_n = 0$, so that
$B_n\wedge C_n \leq1\{T\not\subseteq\widehat T\} B_n$.

Next, we bound $B_n$. Note that by the optimality of $\hat
\beta$ in the {Lasso} problem, and letting $\widehat\delta= \widehat
\beta- \beta_0$,
%
%
\begin{equation}\label{endarray}
B_n := \widehat Q(\hat\beta) - \widehat Q(\beta_0 ) \leq\frac
{\lambda}{n}( \| \beta_0\|_{1} - \|\hat\beta\|_{1}) \leq\frac
{\lambda
}{n}( \| \widehat\delta_T \|_{1} - \|\widehat\delta_{T^c}\|_{1}).
\end{equation}
If $\|\widehat\delta_{T^c}\|_{1} > \|\widehat\delta_{T}\|_{1}$, then
we have
$ \hat Q(\hat\beta) - \hat Q(\beta_0 ) \leq0$.
Otherwise, if $\|\widehat\delta_{T^c}\|_{1} \leq\|\widehat\delta
_{T}\|_{1}$, then, by \hyperref[conrec]{$\mathit{RE}(1)$}, we have
%
%
\begin{equation}\label{endarray2}
B_n := \widehat Q(\hat\beta) - \widehat Q(\beta_0 ) \leq
\frac{\lambda}{n} \|\widehat\delta_{T}\|_{1} \leq\frac{\lambda}{n}
\frac{\sqrt{s}\|\widehat\delta\|_{2,n}}{\kappa(1)}.
\end{equation}
The result follows by applying Theorem~\ref{Thm:Nonparametric} to bound
$\|\widehat\delta\|_{2,n}$, under the condition that \hyperref[conrec]{$\mathit{RE}(1)$} holds,
along with Theorem~\ref{Thm:2StepMain}.

The second claim follows from the first by using $\lambda\lesssim
\sqrt
{n\log p}$ under Condition \textup{\hyperref[conv]{V}}, the specified conditions on the penalty
level. The final bound follows by applying the relation that
for any nonnegative numbers $a, b$, we have $\sqrt{ab}\leq(a+b)/2$.
\end{pf*}
\end{appendix}
\section*{Acknowledgements}
We thank Don Andrews, Whitney Newey, and Alexandre Tsybakov as well as
participants of the Cowles Foundation Lecture at the 2009 Summer
Econometric Society meeting and the joint Harvard--MIT seminar for
useful comments. We also thank Denis Chetverikov, Brigham Fradsen,
Joonhwan Lee, two refereers, and the associate editor for numerous
suggestions that helped improve the articler. We thank Kengo Kato for
pointing out the usefulness of the approach of \cite
{RudelsonVershynin2008} for bounding sparse eigenvalues of the
empirical Gram matrix. We gratefully acknowledge the financial support
from the National Science Foundation.
%

\begin{supplement}
\stitle{Supplementary material for Least squares after
model selection in high-dimensional sparse models}
\slink[doi]{10.3150/11-BEJ410SUPP}  
\sdatatype{.pdf}
\sfilename{BEJ410\_supp.pdf}
\sdescription{The online supplemental article~\cite{BC-PostL1-SA} contains a finite
sample results for the estimation of $\sigma$, details regarding the
oracle problem, omitted proofs, uniform control of sparse eigenvalues,
and Monte Carlo experiments to access the performance of the estimators
proposed in the paper.\looseness=-1}
\end{supplement}

%
%

\printhistory

\end{document}